\documentclass[1p]{elsarticle}

\usepackage{amsmath,amsthm,amsopn,amstext,amscd,amsfonts,amssymb,mathrsfs,mathtools}
\usepackage{relsize}
\usepackage{dsfont}
\usepackage{comment}
\usepackage{float}
\usepackage[active]{srcltx}
\usepackage{graphicx, mathdots}
\usepackage{mathtools}
\usepackage[active]{srcltx}
\usepackage{graphicx, epsfig, subfig}
\usepackage{lscape}
\usepackage{rotating}
\usepackage{caption}
\usepackage{multicol}
\usepackage{colortbl}

\DeclareMathOperator{\diag}{diag}
\DeclareMathOperator{\sgn}{sgn}

\newcommand{\dps}{\displaystyle}

\newtheorem*{problem}{\sc Question (Q)}

\newtheorem{theorem}{\sc Theorem}[section]
\newtheorem{lemma}{\sc Lemma }[section]
\newtheorem{eje}{\sc Example }[section]
\newtheorem{coro}{\sc Corollary}[section]
\newtheorem{proposition}{\sc Proposition}[section]
\newtheorem{obs}{\sc Observation}[section]

\newcommand{\be}{\begin{equation}}
\newcommand{\ee}{\end{equation}}
\newcommand{\R}{\mathbb{R}}

\newcommand{\la}{\lambda}
\newcommand{\wt}{\widetilde}
\newcommand{\wh}{\widehat}

\newcommand{\Ar}{\mathrm{A}}
\newcommand{\Br}{\mathrm{B}}
\newcommand{\Dr}{\mathrm{D}}
\newcommand{\Jr}{\mathrm{J}}
\newcommand{\Ir}{\mathrm{I}}
\newcommand{\Hr}{\mathrm{H}}
\newcommand{\Sr}{\mathrm{S}}
\newcommand{\Ur}{\mathrm{U}}
\newcommand{\PR}{\mathrm{P}}

\newcommand{\Tr}{\mathrm{T}}
\newcommand{\pr}{\mathrm{p}}
\newcommand{\qr}{\mathrm{q}}

\newcommand{\As}{\mathsf{A}}
\newcommand{\Bs}{\mathsf{B}}
\newcommand{\Cs}{\mathsf{C}}
\newcommand{\Ds}{\mathsf{D}}
\newcommand{\Es}{\mathsf{E}}
\newcommand{\Is}{\mathrm{I}}
\newcommand{\Ns}{\mathrm{N}}

\begin{document}
\title{On variation of eigenvalues of birth and death matrices and random walk matrices}
\author{K. Castillo}
\ead{kenier@mat.uc.pt}
\address{CMUC, Department of Mathematics, University of Coimbra,  3001-501 Coimbra, Portugal}
\author{I. Zaballa}
\ead{ion.zaballa@ehu.eus}
\address{Departamento de Matemática Aplicada y EIO, Euskal Herriko Univertsitatea (UPV/EHU), 
Apdo. Correos 644, Bilbao 48080, Spain}
\date{\today}

\begin{keyword}
Birth and death matrices\sep random walk matrices\sep eigenvalues\sep monotonicity
\MSC[2010]{15A18, 15A42, 65F15}
\end{keyword}

\begin{abstract} 
The purpose of this note is twofold: firstly to improve the known results on variation of extreme
eigenvalues of  birth and death matrices and random walk matrices; and secondly to progress towards the
solution of a thirty years old open problem concerning the variation of eigenvalues of these matrices.
\end{abstract}
\maketitle

\section{Introduction}\label{intro}
This note is essentially concerned with eigenvalue problems for certain tridiagonal matrices whose
origin lies in infinite systems of differential equations describing non-homogeneous birth and death
processes in a population.   These are special cases of Markov processes which, in the homobeneous case,
were introduced by Feller \cite{Feller39} and have since been used as models for population  growth, 
queue  formation,  in  epidemiology  and  in  many  other areas  of  both  theoretical  and  applied interest
(see for example \cite{HEBR06}, \cite{Zeif05}, \cite{ZLOSS06}, \cite{I05}). The fundamental differential
equations of the process can be written in the form (see for example \cite{Zeif05} or
\cite[Th. 5.2.1]{I05}, although the coefficients in the latter are the opposite to the ones shown below)
\[
\frac{\mathrm{d}}{\mathrm{d} t} p(t)= \Hr(t) p(t),
\]
where $\Hr(t)=[h_{ij}(t)]$ is the infinite matrix defined as follows:
\[
h_{ij}(t)=\left\{\begin{array}{cll}
b_{i-1}(t) & \text{if} & j=i-1\\
a_{i-1}(t) & \text{if} & j=i+1\\
-(a_{i-1}(t)+b_{i-1}(t)) & \text{if} & j=i\\
0 & \text{if} & \text{otherwise},
\end{array}\right.
\]
and $a_i(t)$ and $b_i(t)$ are positive  functions (except $b_0(t)$ which may be identically $0$) defined on
a non-degenerate open interval of the real line.

During the {\em Workshop on q-Series and Partitions} held at the University of Minnesota on March $1988$ 
\cite[Problem $2$]{I89} and collected in his more recent monographs \cite[Problem $24.9.2$]{I05}, M. E. H. Ismail
arose  a problem about the zeros of birth and death polynomials and random walk polynomials.
Ismail's problem admits a matrix formulation that is precisely given below. It  is related to
the eigenvalues of the following matrices, for any given positive integer $n$:
\begin{align}\label{At}
\wt{\Ar}(t)&=
\begin{pmatrix}
-(a_0(t)+b_0(t)) & a_0(t) & &\\
\ \ b_1(t) &   -(a_1(t)+b_1(t)) &  \ddots & \\
& \ddots & \ddots &  a_{n-1}(t)\\
& & b_{n}(t) & -(a_n(t)+b_{n}(t))
\end{pmatrix},\\
\label{eq.defB} \mathrm{B}(t)&=
\begin{pmatrix}
0 & c_ 0(t)& & \\
1-c_1(t) &   0 & \ddots& \\
& \ddots & \ddots &  c_{n-1}(t)\\
& & 1-c_{n}(t) & 0
\end{pmatrix},
\end{align}  
where
\begin{equation}\label{eq.defc}
c_{j}(t)=\frac{a_{j}(t)}{a_{j}(t)+b_{j}(t)},\quad j=0, \dots, n.
\end{equation}
When $b_0(t)$ is identically $0$, $\wt{\Ar}(t)$ was called the {\em $(n+1)$th $($complete$)$ section }of $\Hr(t)$
by Ledermann and Reuter in their fundamental paper on birth and
death processes \cite[p. 324]{LR54} (see also \cite[p. 267]{B60}). For notational simplicity, 
instead of $\wt{\Ar}(t)$ of (\ref{At}), we will work with 
\begin{equation}\label{eq.defA}
\Ar(t)=
\begin{pmatrix}
a_0(t)+b_0(t) & a_0(t) & &\\
\ \ b_1(t) &   a_1(t)+b_1(t) &  \ddots & \\
& \ddots & \ddots &  a_{n-1}(t)\\
& & b_{n}(t) & a_n(t)+b_{n}(t)
\end{pmatrix}.
\end{equation}
Notice that each eigenvalue of $\wt{\Ar}(t)$ is the opposite to one eigenvalue of $\Ar(t)$.
In fact,  if $\Ur=\diag(-1, 1,-1,1,\ldots, (-1)^{n+1})$ then $\Ar(t)= \Ur(-\wt{\Ar}(t))\Ur$. We also have
\begin{equation}\label{eq.BmB}
 \Br(t)=\Ur(-\Br(t))\Ur.
\end{equation}
Thus, the general assumptions will be that we are given two differentiable with continuous derivative matrix maps
$t \mapsto \mathrm{A}(t)$ and $t \mapsto \mathrm{B}(t)$, defined by (\ref{eq.defA}) and (\ref{eq.defB}), respectively,
from a non-empty (open) interval of the real field $\mathsf{I}\subseteq\R$
into the space of matrices
$\R^{(n+1)\times(n+1)}$. The differentiable real functions
$a_i(t)$ and $b_i(t)$ are assumed to satisfy the following conditions for $t\in\mathsf{I}$:
\begin{equation}\label{eq.defab}
\begin{array}{l}
b_0(t)\geq 0,\\
b_{j}(t)>0,\quad  j=1, \dots, n,\\
a_j(t)>0,\quad  j=0, \dots, n,
\end{array}
\end{equation}
and $c_i(t)$ is defined in (\ref{eq.defc}). It is also assumed that $a_j(t)=0$ and $b_j(t)=0$ if $j<0$.
Under these conditions, $\mathrm{A}(t)$ and $\mathrm{B}(t)$ are called 
{\em birth and death} and  {\em random walk} matrices, respectively. 

We can precisely state now Ismail's problem  in matrix terms:

\begin{problem}\label{qQ}
Identify, when they exist, those subsets of $\ \mathsf{I}$ at which the eigenvalues of $\mathrm{A}(t)$ and
$\mathrm{B}(t)$ are strictly monotone function of $\ t$.
\end{problem}

As a matter of notation, since most matrices will be square of order  $(n+1)$ over $\R$, we dispense
ourselves with mentioning it unless the contrary is expressly stated with a subscript.  Also, given a matrix
$\mathrm{H}$ with only real eigenvalues, we denote by $\lambda_j(\mathrm{H})$ ($0\leq j \leq n$), or 
simply by $\lambda_j$  when this does not lead to confusion, its eigenvalues arranged in increasing order: 
$$
\lambda_{\min}=\lambda_0\leq \lambda_1\leq \cdots\leq \lambda_n=\lambda_{\max}.
$$

Ismail himself proved some relevant results about the monotonicity of the extreme eigenvalues of $\Ar(t)$ and
$\Br(t)$. We need to introduce the following subsets of $\mathsf{I}$: 
\begin{align}
\label{eq2} \As^{\uparrow}_{\min}&=\Big\{t \in \Is \, | \, a_0'(t)> 0 \text{ \rm and } b_0(t)= 0 \text{ \rm and } \forall j\in\{1,\dots,n\}\\
\nonumber &\quad \; \big(a_j'(t)> 0  \text{ \rm and }  a'_j(t)b_j(t)-a_j(t)b'_j(t)>0 \big)\Big\},\\
\label{eq1} \As^{\uparrow}_{\max}&=\Big\{t \in \Is \, | \, \forall j\in\{0,\dots,n\} \big(a_j'(t)> 0 
\text{ \rm and }  b_j'(t)> 0  \big)\Big\},\\
\label{eq3} \Cs^{\uparrow}_{\max}&=\Big\{t \in \Is\, | \, c_0(t)=1  \text{ \rm and }  
\forall j\in\{1,\dots,n\}\big(c_j'(t)< 0\big)\Big\}.
\end{align}
The sets $\As^{\downarrow}_{\min}$, $\As^{\downarrow}_{\max}$, and $\Cs^{\downarrow}_{\max}$
are defined analogously  by exchanging the roles of $>$ and $<$.  Ismail proved in \cite[Th. $1$ and $2$]{I87}
(see also \cite[Th. $2.2$]{I89} and \cite[Th. $7.4.2$]{I05}) that $\la_{\min}(\Ar)$ (resp., 
$\la_{\max}(\Ar)$) is a strictly increasing function of $t$ in each one of the non-degenerate subintervals  of
$\mathsf{A}^{\uparrow}_{\min}$ (resp., $\mathsf{A}^{\uparrow}_{\max}$). He also proved
\cite[Th. $3$]{I87} (see also \cite[Th. $2.3$]{I89} and \cite[Th. 7.4.3]{I05}) that
$\la_{\max}(\Br)$ is a strictly increasing function of $t$ in each one of the non-degenerate subintervals of
$\mathsf{C}^{\uparrow}_{\max}$. 

One of the goals of this paper is to give new and wider subsets where the extreme eigenvalues of $\Ar(t)$ 
 monotonically increase or decrease. This is done in Section \ref{ext}. On the other hand, Magagna, in his
 Ph. D. Thesis of 1965  and \cite{HM70} also addressed the problem of the
monotonicity of the eigenvalues of birth and death matrices. To be precise, the eigenvalues of birth and death
matrices are real and simple (see Section \ref{sec.prelim}) and so they are differentiable functions of the matrix
coefficients. A thorough analysis of this dependence, in the case of  homogeneous (i.e.;  time-independent)
birth and death matrices with $b_0=a_n=0$, allowed Magagna \cite[Result 2.2, p. 2-11]{M65}  and Horne and Magagna
\cite[Theorem $1$]{HM70} to derive directions in $\R^{2n}$ on which the eigenvalues strictly increase. Specifically,
assume that matrix $\Ar$ of (\ref{eq.defA}) is constant  with $a_n=b_0=0$. Look at the nonzero entries of this
matrix as real parameters. Thus $\Ar$ is a matrix depending of $2n$ real variables. Observe that the assumption
$b_0=a_n=0$ implies that $\Ar$ is a singular matrix.

\begin{theorem}\label{thm.Magagna}
If $\ a_{j-1},\ b_j>0$ for $j=1,2,\ldots, n$ and $r>0$ then the nonzero eigenvaues of $\Ar$ are strictly increasing 
along the half lines $b_{i+1}=r a_i$, $i=0,1,\ldots, n-1$, and $b_i=ra_i$, $i=1,\ldots, n-1$.
\end{theorem} 

We will show in Section \ref{sec.QueQ} how to apply and generalize this result to  the time-dependent
birth and death matrices of (\ref{eq.defA}) in order to tackle Question Q above . Finally, we will deal in
Section \ref{sec.randwalkmat} with the monotonicity of the eigenvalues of the random walk matrices of
(\ref{eq.defB}). It will be seen that there is a very close relationship between the eigenvalues of these
matrices and certain birth and death matrices constructed with their elements. This relationship will allow
to apply to random walk matrices all results obtained for birth and death matrices in the previous sections.
Preliminary notions and auxiliary results are collected in Section \ref{sec.prelim}. 

\section{Preliminaries}\label{sec.prelim}
This section is devoted to review some spectral properties of matrices $\Ar(t)$ and $\Br(t)$ of
(\ref{eq.defA}) and (\ref{eq.defB}). The main reference for the results to follow  is \cite{GK41}.
For each $t\in\Is$,  $\Ar(t)$ is a Jacobi matrix (see \cite[Ch. II, Sec. 1]{GK41}) and so, its eigenvalues
are real and distinct. A consequence of this property  and that $\Ar(t)$ depends differentiably
on $t\in\Is$ is that the eigenvalues of $\Ar(t)$ are
differentiable functions of $t$ (see, for example, \cite[p. 102]{L97} or \cite[p. 183]{SS04}).
They can be arranged in increasing order:
\[
\la_0(\Ar,t)<\la_1(\Ar,t)<\cdots <\la_n(\Ar,t).
\]
In addition, if  for $k=1,2,\ldots, n+1$, $\Ar(1:k,1:k)(t)$ denotes the principal submatrix of $\Ar(t)$ 
formed by its $k$ first rows and columns, the eigenvalues of $\Ar(1:k,1:k)(t)$ and $\Ar(1:k-1,1:k-1)(t)$
interlace (see \cite[Ch. II, Sec. 1]{GK41}). That is to say, for each $t\in \Is$ and $ j=1,2,\ldots, k$:
\be\label{eq.interlacing}
\la_{j-1}(\Ar(1:k,1:k),t)<\la_{j-1}(\Ar(1:k-1,1:k-1),t)<\la_j(\Ar(1:k,1:k),t).
\ee
Next, let $\Delta_0(t)=1$ and  $\Delta_k(t)=\det \Ar(1:k,1:k)(t)$. It is easily seen by induction
on $k$ that for $t\in\Is$,
\[
\Delta_k(t)=a_{k-1}(t)\Delta_{k-1}(t)+\prod_{j=0} ^{k-1} b_j(t),\quad k=1,\ldots, n+1.
\]
Henceforth $\Delta_k(t)>0$ for $k=1,\ldots, n+1$. It follows from \cite[Ch. II, Th. 10 ]{GK41}
that $\Ar(t)$ is  an oscillatory matrix and then, all its eigenvalues are positive \cite[Ch. II, Th. 6]{GK41}: 
\be\label{eq.poseig}
0<\la_0(\Ar,t)<\la_1(\Ar,t)<\cdots <\la_n(\Ar,t),\quad t\in\Is.
\ee
Although seeing the birth and death matrices as oscillatory matrices is convenient for our developments
it is worth-pointing out that they are also diagonally dominant matrices. Since they are diagonally similar to
symmetric matrices with positive diagonal elements (see \eqref{eq.defS}), it follows
from a result by Taussky (see \cite[Cor. 6.2.27]{HJ}) that all their eigenvalues are positive.

On the one hand, $\Br(t)$ is also a Jacobi matrix but it is not an oscillatory matrix because it is not
totally non-negative (i.e., all minors are not non-negative). However,  $\wh{\Ar}(t)= \Ir_{n+1}+\Br(t)$ is a birth
and death matrix with $\wh{a}_i(t)=c_i(t)$ and $\wh{b}_i(t)=1-c_i(t)$, $t\in\Is$, $i=0,1,\ldots, n$.
Since, for each $t\in\Is$, $\la_i(\Br,t)=\la_i(\wh{\Ar},t)-1$, the eigenvalues of $\Br(t)$ are also real and simple,
and they are differentiable functions of $t\in\Is$. It follows from (\ref{eq.poseig}) that
\[
-1<\la_0(\Br,t)<\la_1(\Br,t)<\cdots <\la_n(\Br,t),\quad t\in\Is.
\]
But by  (\ref{eq.BmB}), for each $t\in\Is$, the eigenvalues of $\Br(t)$ are symmetrically distributed with respect to the origin.
Hence
\be\label{eq.eigB}
-1<\la_0(\Br,t)<\la_1(\Br,t)<\cdots <\la_n(\Br,t)<1,\quad t\in\Is,
\ee
half of them being positive and the other half negative. Moreover, if $n$ is even then $0$ is an
eigenvalue of $\Br(t)$ for all $t\in\Is$, implying that, when $n$ is even, $\det \Br(t)=0$.

As far as the eigenvectors are concerned, since the eigenvalues of $\Ar(t)$ are simple, each eigenvalue $\la_k(A,t)$
admits an eigenvector $u_k(t)$ which depends differentiably on $t$ (see \cite[Ch. 9, Th. 8]{L97}). In
addition (see \cite[Cap. II, Th. 6]{GK41}) \textit{among the coordinates of $u_k(t)$ there are exactly
$k-1$ sign changes}. The same properties apply to the eigenvalues of $\Br(t)$. This is a general result for the
eigenvectors of matrices depending differentiably on $t\in\Is$. However, for $\Ar(t)$ and $\Br(t)$
explicit expressions of some distinguished eigenvectors can be given.
Specifically, for each $t\in\Is$ let $\{p_k(x;t)\}$ be the family of (orthogonal)  polynomials defined recursively as follows:
 \begin{align}
\nonumber p_{-1}(x;t)&=0,\quad p_0(x;t)=1,\\
\label{eq.delpeivec}xp_k(x;t)&=\alpha_k(t)p_{k+1}(x;t)+\beta_k(t)p_{k}(x;t)+\gamma_k(t)p_{k-1}(x;t),\quad k=0,1,\dots,
\end{align}
where, for $i=0,1,2,\ldots$,  $\alpha_i(t)$, $\beta_i(t)$, and $\gamma_i(t)$ are differentiable functions of $t\in\Is$
and $\alpha_{i-1}(t)\gamma_i(t)>0$.
We can associate to this family of polynomials the following infinite Jacobi matrix:
\[
\Hr(x;t)=\begin{pmatrix}
\beta_0(t)-x & \alpha_0(t) & & &\\
\gamma_1(t) & \beta_1(t)-x & \alpha_1(t) & &\\
 &\ddots & \ddots&\ddots & \\
  & &  \gamma_n(t) & \beta_n(t)-x & \alpha_ n(t)\\
  &&\ddots&\ddots & \ddots \\
\end{pmatrix},\quad \pr(x;t)=\begin{pmatrix}
p_0(x;t)\\p_1(x;t)\\\vdots \\p_n(x;t)\\\vdots
\end{pmatrix}.
\]
Observe that the submatrix formed by the first $k$ rows and columns of $\Hr(x;t)$ is $\Jr_k(t) -x I_{k}$ where
\[
\Jr_{k}(t)=\begin{pmatrix} \beta_0(t)& \alpha_0(t) & & \\
\gamma_1(t) & \beta_1(t) & \alpha_1(t) &\\
 &\ddots & \ddots&\ddots  \\
  & &  \gamma_{k-1}(t) & \beta_{k-1}(t) 
 \end{pmatrix}.
\]
is a finite Jacobi matrix of order $k$. The following result is well-known and can be easily proven using induction, for example.
\begin{proposition}\label{prop.polJac}
With the above notation, for all $t\in\Is$,
\begin{itemize}
\item[(i)] $\Hr(x;t)\,\pr(x;t)=0$,
\item[(ii)] $p_k(x;t)=\dps\frac{(-1)^{k}}{\alpha_0(t)\alpha_1(t)\cdots \alpha_{k-1}(t)}\,\det (J_k(t)-xI_{k})$, $\quad k=1,2,\ldots.$
\end{itemize}
\end{proposition}
In other words, for each $k=1,2,\ldots$ and each $t\in\Is$ the eigenvalues of $\Jr_k(t)$ are the roots
of $p_k(x;t)$ and if $\la_0(t)$ is an eigenvalue of $\Jr_k(t)$ then $\pr_k(t)=(p_0(\la_0(t);t), p_1(\la_0(t);t),$ $\dots, p_{k-1}(\la_0(t);t))^\mathsf{T}$ is an eigenvector of $\Jr_k(t)$ for
$\la_0(t)$. Since the eigenvalues of $\Jr_k(t)$ are simple, $\la_0(t)$ differentiably depends on $t$ and so does
$\pr_k(t)$.

All above directly applies to $\Ar(t)$ and $\Br(t)$. In addition, since the non-diagonal entries of these
matrices are positive, their eigenvectors satisfy the following important property (see \cite[Ch. II, Th. 1]{GK41}):
\be\label{eq.sigchan}
\begin{array}{l}
\text{\textit{For each $t\in \Is$ the sequence of coordinates of the eigenvectors of the $j$-th}}\\ 
\text{\textit{eigenvalue of $\Ar(t)$ and $\Br(t)$ has exactly $j-1$ sign changes.}}
\end{array}
\ee
In particular, all coordinates of the eigenvectors for $\la_{\max}$ have the same signs and the signs
of the coordinates of the eigenvectors for $\la_{\min}$ alternate. This property will be useful in Section \ref{ext}.

We close this section with a well-known formula for the derivatives of the eigenvalues of $\Ar(t)$ and $\Br(t)$
(see, for example, \cite[Ch. 9]{L97}). For $k=0,1,\ldots,n$ 
\be\label{eq.deriveig}
\la_k^\prime(\Ar,t)=\frac{ y_k^\mathsf{T}(t) \Ar'(t) x_k(t)}{y_k^\mathsf{T}(t) x_k(t)},\quad t\in\Is,
\ee
where $x_k(t)$ and $y_k(t)$ are right and left eigenvectors of $\Ar(t)$ for the eigenvalue $\la_k(\Ar,t)$;
that is, $\Ar(t)x_k(t)=\la_k(\Ar,t) x_k(t)$ and $y^\mathsf{T}_k(t) \Ar(t)=\la_k(\Ar,t) y^\mathsf{T}_k( t)$ for each $t\in\Is$.

\section{Extreme eigenvalues of birth and death matrices}\label{ext}
When dealing with specific matrices, even for rather simple ones, the sets in (\ref{eq1}) or (\ref{eq2}) may
provide poor or none information about the intervals where the actual extreme eigenvalues of the birth and dead
matrices increase or decrease. The following example is an illustration.


\begin{eje}\label{Meixner}{\rm
Consider the $3$-by-$3$  birth and death matrix
$$
\mathrm{A}_1(t)=\begin{pmatrix}
 \dps\frac{1}{t}+\frac{1}{1-t} & \dps\frac{1}{t} & 0\\[7pt]
 t & 1 & 1-t\\[7pt]
 0 & \dps \frac{1}{1-t} & \dps\frac{1}{t}+\dps\frac{1}{1-t}
\end{pmatrix}, \quad t\in (0,1).
$$ 
The eigenvalue functions of this matrix are depicted in Figure \ref{FigEx22}. They can be explicitly computed in
this example. In particular, the second eigenvalue-function (the one in red in the Figure) is
$\la_2(t)=1/t+1/(1-t)$. It is decreasing in the interval $(0,1/2)$ and increasing in $(1/2,1)$. 
It can be seen (using software, for instance) that $\la_1(t)$ and $\la_3(t)$ are also decreasing in $(0,1/2)$
and increasing in $(1/2,1)$ approximately. However, $a'_i(t)<0$ while $b'_i(t)>0$ in $\Is$ for $i=0,1,2$.
Therefore $\mathsf{A}^{\uparrow}_{\max}= \mathsf{A}^{\downarrow}_{\max} 
=\emptyset$. But also $\As_{\min}^{\uparrow}=\As_{\min}^{\downarrow}=\emptyset$ because $b_0(t)\neq 0$.
Notice that for all $t\in(0,1)$, $a'_i(t)<0$ and $a'_i(t)b_i(t)-a_i(t)b'_i(t)<0$ for $i=0,1,2$.  However,
$\la_{\min}(t)$ is not decreasing in the whole interval $(0,1)$.
}
\begin{figure}[H]
\centering
 \includegraphics[width=9cm]{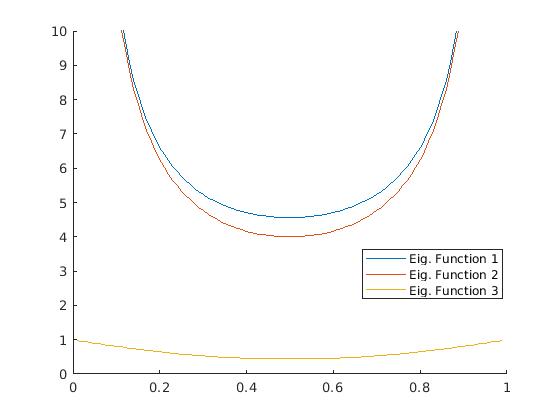}
  \caption{Behavior of the eigenvalues of $\mathrm{A}_1(t)$.}\label{FigEx22}
\end{figure} 
\end{eje}

The above example illustrates how far the set of conditions that characterize
$\mathsf{A}^{\uparrow}_{\max}$, $\mathsf{A}^{\downarrow}_{\max}$, $\As_{\min}^{\uparrow}$
and $\As_{\min}^{\downarrow}$ can be
from being necessary conditions for the monotonicity of the extreme eigenvalues
of $\mathrm{A}(t)$. In this section we aim to provide wider sets where the extreme eigenvalues
$\la_{\max}(t)$ and $\la_{\min}(t)$ of $\Ar(t)$ increase and decrease. We will use the
fact that $\Ar(t)$ can be symmetrized by means of a diagonal similarity transformation.
In fact, let $\mathrm{D}(t)=\diag(d_0(t), d_1(t),\ldots, d_n(t))$, where 
\be\label{eq.defdj}
d_0(t)=1, \quad d_j(t)=\sqrt{\frac{a_0(t) \dots a_{j-1}(t)}{b_1(t)\cdots b_j(t)}}, \quad j=1,\dots, n.
\ee
Observe that for $j=0,1,\ldots, n$, $d_j(t)$ is a well-defined positive function because $b_j(t)>0$ and
$a_j(t)>0$ for all $t\in \Is$.
An easy computation shows that 
\be\label{eq.defS}
\begin{array}{rcl}
\mathrm{S}(t)&=&\Dr(t) \Ar(t)\Dr(t)^{-1}\\[7pt]
&=&\begin{pmatrix}
a_0(t)+b_0(t)& \sqrt{a_0(t) b_1(t)} & & \vspace*{.5em}\\
 \sqrt{a_0(t) b_1(t)} &  \ \ a_1(t)+b_1(t) &   \ddots  \\
 & \ddots &  \ddots & \sqrt{a_{n-1}(t)b_n(t)}\vspace*{.5em}   \\
 & & \sqrt{a_{n-1}(t)b_n(t)} & a_{n}(t)+b_n(t)
\end{pmatrix}.
\end{array}
\ee
This is a tridiagonal, symmetric matrix with the same eigenvalues as $\Ar(t)$ for each $t\in\Is$. So,
we can use $\Sr(t)$ to compute the subsets of $\ \Is$ where the eigenvalue-functions of $\Ar(t)$
increase or decrease. Since $\Sr(t)$ is symmetric, if $\la_k(t)$ is one of its eigenvalues and
$x_k(t)$ is a right eigenvector then $x_k(t)$ is also a left eigenvector for $\la_k(t)$. On the other hand,
it follows from item (i) of Proposition \ref{prop.polJac} that if for each $t\in\Is$ we define
\begin{align}
\nonumber q_{-1}(\la_k(t);t)&=0,\quad q_0(\la_k(t);t)=1,\\
\label{eq.defqlk} \la_k(t)q_j(\la_k(t);t)&=\sqrt{a_j(t)b_{j+1}(t)}q_{j+1}(\la_k(t);t)+(a_j(t)+b_j(t))(t)q_{j}(\la_k(t);t)\\
\nonumber &\quad +\sqrt{a_{j-1}(t)b_j(t)}q_{j-1}(\la_k(t);t), \quad j=0,1,\dots,
\end{align}
then (cf. (\ref{eq.delpeivec}))
\begin{align}\label{eq.qk}
\qr_{k}(t)=( q_0(\la_k(t);t), q_1(\la_k(t);t), \cdots, q_{n}(\la_k(t);t))^\mathsf{T}
\end{align}
is an eigenvector of $\Sr(t)$ for $\la_k(t)$. Bearing in mind   (\ref{eq.deriveig}) and the fact that $\qr_{k}^\mathsf{T}(t)\qr_{k}(t)>0$, we are
to find the values of  $\ t\in \Is$
where $\qr_{k}^\mathsf{T}(t)\Sr'(t)\qr_{k}(t)$ is strictly positive or negative.  Let us compute this function.
For notational simplicity we remove the subscript $k$ and the dependence on $t$ and $\la_k$.
\begin{align}
\label{eq.qtS'q}\qr^\mathsf{T}\Sr'\qr&= \sum\limits_{j=0}^{n-1}(a'_j+b'_j)q_j^2+2\sum\limits_{j=0}^{n-1} (\sqrt{a_{j}b_{j+1}})' q_{j}q_{j+1}\\
\nonumber &= (a'_0+b'_0)+(\sqrt{a_0b_1})' q_1+\sum\limits_{j=1}^{n-1} r_j q_j+((\sqrt{a_{n-1}b_n})' q_{n-1}+(a'_n+b'_n)q_n)q_n,
\end{align}
where
\be\label{eq.defrj}
r_j=(\sqrt{a_{j-1} b_j})'\, q_{j-1}+(a'_j+b'_j) \, q_{j}+(\sqrt{a_{j} b_{j+1}})'\, q_{j+1}.
\ee
Using (\ref{eq.defqlk}) to compute $q_{j+1}$ in terms of $q_j$ and $q_{j-1}$ for $0\leq j\leq n-1$, and
substituting in $r_jq_j$ we get
\[
r_jq_j=\left((a'_j+b'_j)+\frac{\left(\sqrt{a_{j}b_{j+1}}\right)'}{\sqrt{a_jb_{j+1}}}
\left(\la_k-(a_j+b_j)\right)\right)q_j^2+
\sqrt{a_jb_{j+1}}\left(\frac{\sqrt{a_{j-1}b_j}}{\sqrt{a_jb_{j+1}}}\right)' q_{j-1}q_j.
\]
Also, it follows from (\ref{eq.defqlk}) that $q_1=\dps\frac{\la_k-(a_0+b_0)}{\sqrt{a_0b_1}}$. Thus,
\be\label{eq.j0}
(a'_0+b'_0)+(\sqrt{a_0b_1})' q_1=(a'_0+b'_0)+\frac{\left(\sqrt{a_{0}b_{1}}\right)'}{\sqrt{a_0b_{1}}}
\left(\la_k-(a_0+b_0)\right).
\ee
On the other hand, by item (ii) of Proposition \ref{prop.polJac}, $q_{n+1}(\la_k)=0$. Thus,
for $j=n$ in (\ref{eq.defqlk}),
\be\label{eq.qnm10}
q_{n-1}=\frac{1}{\sqrt{a_{n-1}b_n}}(\la_k-(a_n+b_n))q_n,
\ee
and so
\be\label{eq.jn}
((\sqrt{a_{n-1}b_n})' q_{n-1}q_n+(a'_n+b'_n)q_n^2=\left((a'_n+b'_n)+
\frac{\left(\sqrt{a_{n-1}b_{n}}\right)'}{\sqrt{a_{n-1}b_{n}}}(\la_k-(a_n+b_n))\right) q_n^2.
\ee
In conclusion (recall that $a_{-1}(t)=0$ and $q_{-1}(x;t)=0$),
\begin{align}
\label{eq.qtS'qsimpli}
\qr^\mathsf{T}\Sr'\qr&= \sum\limits_{j=0}^{n-1}\left(\left((a'_j+b'_j)+\ell_j\left(\la_k-(a_j+b_j)\right)\right)q_j^2+
\sqrt{a_jb_{j+1}}\ \left(\sqrt{e_j}\right)'\ q_{j-1}\ q_j\right)\\
\nonumber &\quad +\left((a'_n+b'_n)+\ell_{n-1}(\la_k-(a_n+b_n))\right) q_n^2,
\end{align}
where
\be\label{eq.defefgl}
e_j(t)=\frac{a_{j-1}(t)b_j(t)}{a_{j}(t)b_{j+1}(t)}, \quad  
\ell_j(t)=\frac{\left(\sqrt{a_{j}(t)b_{j+1}(t)}\,\right)'}{\sqrt{a_{j}(t)b_{j+1}(t)}},\quad j=0, \dots, n-1.
\ee

\begin{theorem}\label{lg}
Let $\Ar(t)$ be the birth and death matrix of (\ref{eq.defA}) and let $e_j(t)$, 
$\ell_j(t)$, $j=0,1,\ldots, n-1$, be the functions of (\ref{eq.defefgl}). For each $t\in\Is$ set
\begin{align}
\label{eq.sigrom}\dps{m_1(t)=\max_{0\leq j\leq n}\{a_j(t)+b_j(t)\}, \quad
m_2(t)=\min\{\sigma(t),\rho(t)\}},
\end{align}
where
\begin{align*}
\sigma(t)&=\max\Big\{2a_0(t)+b_0(t),a_n(t)+2b_n(t), 2\max_{1\leq i\leq n-1} a_i(t)+b_i(t)\Big\},\\
\qquad \quad \rho(t)&=\max\Big\{a_0(t)+b_0(t)+\sqrt{a_0(t)b_1(t)}, a_n(t)+b_n(t)+\sqrt{a_{n-1}(t)b_n(t)},\\
&\quad \max_{1\leq i\leq n-1} a_i(t)+b_i(t)+\sqrt{a_{i-1}(t)b_i(t)}+\sqrt{a_i(t)b_{i+1}(t)}\Big\}.
\end{align*}
For $j=0,1,\ldots, n-1$ let
 \begin{align*}
 f_j(t)&=a'_j(t)+b'_j(t)+\ell_{j}(t)(m_1(t)-a_j(t)-b_j(t)),\\
 g_j(t)&=a'_j(t)+b'_j(t)+\ell_{j}(t)(m_2(t)-a_j(t)-b_j(t)),\\
f_n(t)&=a'_n(t)+b'_n(t)+\ell_{n-1}(t)(m_1(t)-a_n(t)-b_n(t)),\\
g_n(t)&=a'_n(t)+b'_n(t)+\ell_{n-1}(t)(m_2(t)-a_n(t)-b_n(t)).
 \end{align*}
Define the following subsets of $\Is$:
\begin{align}
\label{eq.bjmaxu} \Bs^{\uparrow}_{j,\max}&=
\Big\{t\in\Is\mid \Big((a_{j-1}(t)b_{j}(t))'\geq 0 \text{ \rm and }
 (a_{j}(t)b_{j+1}(t))'\geq 0 \text{ \rm and }
a'_j(t)+b'_j(t) \geq 0 \Big)\\
\nonumber&\quad \left. \text{ \rm or } \Big((a_{j}(t)b_{j+1}(t))'\leq 0 \text{ \rm and } g_j(t)>0
\text{ \rm and } e'_j(t)\geq 0\Big)\right.\\
\nonumber&\quad \left. \text{ \rm or } \Big( 0\geq a'_j(t)+b'_j(t)\text{ \rm and } f_j(t)\geq 0
\text{ \rm and } e'_j(t)\geq 0\Big)\right\},\quad j=0,1,\ldots, n-1,\\
\nonumber \Bs^{\uparrow}_{n,\max}&=\Big\{t\in\Is\mid  \Big((a_{n-1}(t)b_{n}(t))'\geq 0 \text{ \rm and }
a'_n(t)+b'_n(t) \geq 0\Big)\\
\nonumber&\quad\, \left.\text{ \rm or } \Big((a_{n-1}(t)b_{n}(t))'\leq 0 \text{ \rm and } g_n(t)>0\Big)\right.  \text{ \rm or } \Big(0\geq  a'_n(t)+b'_n(t)\text{ \rm and } f_n(t)\geq 0\Big)\Big\}.\\
\label{eq.Nmaxu}\Ns&=\Big\{t\in\Is\mid \exists j\in\{0,1,\ldots, n-1\} \text{ \rm such that }
 (a_{j}(t)b_{j+1}(t))'\neq 0 \text{ \rm or }\\
\nonumber&\quad \;  \exists j\in\{0,1,\ldots, n\} \text{ \rm such that } a'_j(t)+b'_j(t)\neq 0\Big\}.
\end{align}
The sets $\Bs^{\downarrow}_{j,\max}$ are defined analogously  by exchanging the roles of $>$ and $<$ on the one hand,
 and $\geq$ and $\leq$ on the other hand.
Let $\Bs^{\uparrow}_{\max}=\left(\bigcap_{j=0}^n \Bs^{\uparrow}_{j,\max}\right)\bigcap \Ns$ and 
$\ \Bs^{\downarrow}_{\max}=\left(\bigcap_{j=0}^n \Bs^{\downarrow}_{j,\max}\right)\bigcap \Ns$.
Then $\la_{\max}(\Ar,t)$ is a strictly increasing (resp., strictly decreasing) function of $t$ in each one of the
non-degenerate subintervals of $\mathsf{B}^{\uparrow}_{\max}$
(resp., $\mathsf{B}^{\downarrow}_{\max}$).
\end{theorem}

\medskip
\begin{proof}
For each $t\in\Is$, all entries of $\Sr(t)$ are non-negative and $\Sr(t)$ is irreducible. The latter means that
there is no permutation matrix $\PR$ such that $\PR^\mathsf{T}\Sr(t) \PR=\begin{psmallmatrix} \Sr_1(t) &  \Sr_2(t)\\
0&\Sr_3(t)\end{psmallmatrix}$. By Perron-Frobenious Theorem (\cite[Ch. $8$]{Se10} or
\cite[Th. 1.4.4]{BR97}) for each $t\in \Is$, there is a positive eigenvector of $\Sr(t)$ for
its biggest eigenvalue $\la_{\max}(t)>0$. Since $\qr_n(t)$ of (\ref{eq.qk}) is an eigenvector of  $\la_{\max}(t)$, 
 $q_0(\la_{\max}(t);t)=1$ and by (\ref{eq.sigchan}) all coordinates of
 the eigenvectors of $\la_{\max}(t)$ have the same sign, we conclude that $\qr_n(t)>0$.
We are to prove that if $t\in \Bs^{\uparrow}_{\max}$ then $\qr_n^\mathsf{T}(t)\Sr'(t)\qr_n(t)>0$. 
The proof that if $t\in \Bs^{\downarrow}_{\max}$ then $\qr_n^\mathsf{T}(t) S'(t) \qr_n(t)<0$ is
similar.  As above, we remove the dependences on $t$ and $\la_{\max}$ for notational simplicity
and consider that $t\in\Is$ has been fixed.  First of all, we are to show that $m_1< \la_{\max}\leq m_2$. In fact, since $\Sr$ is symmetric,
if $a_h+b_h=m_1$ then $$\la_{\max}=\max_{\|u\|_2=1}u^\mathsf{T} \Sr  u\geq e_{h}^\mathsf{T} \Sr e_{h}=
m_1.$$ But, one can prove using  (\ref{eq.interlacing}) and (\ref{eq.defqlk}) that, actually,
$\dps\la_{\max}>m_1$. On the other hand, by Ger\u{s}gorin's Theorem (see for example 
\cite[Th. 6.1.1]{HJ}) applied to matrix $\Ar$, 
$$\la_{\max}\in[b_0,2a_0+b_0]\bigcup
\left(\bigcup_{i=1}^{n-1} [0,2(a_i+b_i)]\right)\bigcup [a_n,a_n+2b_n].$$ And, applied to $\Sr$, 
\begin{align*}
\la_{\max}\in&\left[a_0+b_0-\sqrt{a_0b_1}, a_0+b_0+\sqrt{a_0b_1}\right]\\
&\bigcup\left(\bigcup_{i=1}^{n-1} \left[a_i+b_i-(\sqrt{a_{i-1}b_i}+\sqrt{a_ib_{i+1}}), a_i+b_i+(\sqrt{a_{i-1}b_i}+\sqrt{a_ib_{i+1}})\right]\right)\\
&\bigcup \left[a_n+b_n-\sqrt{a_{n-1}b_n}, a_n+b_n+\sqrt{a_{n-1}b_n}\right].
\end{align*}
Henceforth $\la_{\max}\leq \min\{\sigma,\rho\}=m_2$.

Let 
\be\label{eq.qSqn}
(\qr^\mathsf{T}\Sr'\qr)_n=\left((a'_n+b'_n)+\ell_{n-1}(\la_{\max}-(a_n+b_n))\right) q_n^2,
\ee
and for $j\in\{0,1,\ldots, n-1\}$, let 
\be\label{eq.qSqj}
(\qr^\mathsf{T}\Sr'\qr)_j=\left(\left((a'_j+b'_j)+\ell_j\left(\la_{\max}-(a_j+b_j)\right)\right)q_j^2+
\sqrt{a_jb_{j+1}}\ \left(\sqrt{e_j}\right)'\ q_{j-1}\ q_j\right).
\ee
Let  $j$ be any nonnegative integer smaller than $n+1$ and let $t\in \Bs^{\uparrow}_{j,\max}$.
\begin{itemize}
\item Assume that $j<n$ and $(a_{j-1}b_{j})'\geq 0$, $(a_jb_{j+1})'\geq 0$ and $a'_j+b'_j\geq 0$.
If $j=0$ then, by (\ref{eq.j0}),  $(\qr^\mathsf{T}\Sr'\qr)_0=(a'_0+b'_0)+(\sqrt{a_0b_1})' q_1\geq 0$. If $j\in\{1,2,\ldots, n-1\}$
then  it follows from (\ref{eq.defrj}) that $r_j\geq 0$ and so $(\qr^\mathsf{T}\Sr'\qr)_j=r_jq_j\geq 0$. Finally,
if  $(a_{n-1}b_{n})'\geq 0$ and $a'_n+b'_n\geq 0$ then, by (\ref{eq.jn}),
$(\qr^\mathsf{T}\Sr'\qr)_n=(\sqrt{a_{n-1}b_n})' q_{n-1}q_n+(a'_n+b'_n)q_n^2\geq 0$.  Therefore,
$(\qr^\mathsf{T}\Sr'\qr)_j\geq 0$ for $j=0,1,\ldots, n$.

\item Let $j<n$ and assume $(a_{j}b_{j+1})'\leq 0$, $g_j>0$ and $e'_j\geq 0$. 
Note that $e_0=0$. Since $(a_{j}b_{j+1})'\leq 0$, it must be  $\ell_j\leq 0$. Also, it follows from
 $\la_{\max}\leq m_2$ and $\ell_j\leq 0$ that $ a'_j+b'_j+\ell_j(\la_ {\max}- (a_j+b_j))
\geq a'_j+b'_j+\ell_j(m_2-a_j-b_j)=g_j> 0$.  Now, $(\qr^\mathsf{T}\Sr'\qr)_j> 0$ follows from the assumption
$e'_j\geq 0$.


\item Assume now that $j<n$, $0\geq a'_j+b'_j$, $f_j\geq 0$ and $e'_j\geq 0$.
It follows from this assumption that $\ell_j\geq 0$. Using this fact and $\la_{\max}> m_1$ we get
$a'_j+b'_j+\ell_j(\la_{\max}-(a_j+b_j))>a'_j+b'_j+\ell_j(m_1- (a_j+b_j))=
f_j\geq 0$. Bearing in mind that $e'_j\geq 0$, we conclude that
$(\qr^\mathsf{T} \Sr' \qr)_j >0$.

\item If $j=n$ and $(a_{n-1}b_{n})'\leq 0$ and $g_n>0$, or, $j=n$ and
$0\geq a'_n+b'_n\geq \ell_{n-1}(a_n+b_n- m_1)$, then similar arguments to those used in
the previous items allow to prove that $(\qr^\mathsf{T} \Sr'\qr)_n> 0$.
\end{itemize}

Summarizing, if $t\in \Bs^{\uparrow}_{j,\max}$ then for all $j=0,1,\ldots, n$, $(\qr^\mathsf{T} \Sr' \qr)_j\geq 0$.
Moreover,  $(\qr^\mathsf{T} \Sr' \qr)_j> 0$ unless $(a_{j-1}b_j)'=(a_{j}b_{j+1})'=a'_j+b'_j=0$. Henceforth the theorem follows.
\end{proof}

The proof of the following corollary is straightforward.
 \begin{coro}\label{coro.impcond}
 Let $\Ar(t)$ be the birth and death matrix of (\ref{eq.defA}). Then
 $\mathsf{A}^{\uparrow}_{\max}\subset \mathsf{B}^{\uparrow}_{\max}$ and
 $\mathsf{A}^{\downarrow}_{\max}\subset \mathsf{B}^{\downarrow}_{\max}$.
\end{coro}

\begin{obs}\label{obs.lmax}
\begin{itemize}
\item[(i)]{\rm For the matrix $\Ar_1(t)$ of Example \ref{Meixner}, $a_{j-1}(t)b_j(t)=1$ for $j=0,1,2$,
$$a'_0(t)+b'_0(t)=a'_2(t)+b'_2(t)=- \frac{1}{t^2}+\frac{1}{(1-t)^2}$$ and $a'_1(t)+b'_1(t)=0$. Thus 
\begin{align*}
\Bs^{\uparrow}_{\max}&=\left\{t\in(0,1)\,|\, \frac{1}{(1-t)^2}-\frac{1}{t^2}> 0\right\}= (1/2,1),\\
\Bs^{\downarrow}_{\max}&=\left\{t\in(0,1)\,|\, \frac{1}{(1-t)^2}-\frac{1}{t^2}< 0\right\}= (0,1/2).
\end{align*}
This is what the graphic of $\la_{\max}(t)$ in  Figure \ref {FigEx22} shows. This is a \textit{toy example} where the sufficient conditions of Theorem \ref{lg} completely determine the
monotonicity of the biggest eigenvalue of a birth and death matrix. One cannot expect that, in general, such
an accuracy can be derived for those sufficient conditions as the following example shows.

\begin{eje}\label{eje.A2}{\rm
Consider the following birth and death matrix:
\[
\Ar_2(t)=\begin{pmatrix}t & (1-t)t & 0\\[7pt] 
\dps\frac{1}{1-t} & \dps\frac{1}{t}+\dps\frac{1}{1-t} & \dps \frac{1}{t}\\[7pt]
0 & \dps t^2 & \dps t^2+t\end{pmatrix},\quad t\in (0,1).
\]
 Figure \ref{FigExorig} depicts
the graphics of its eigenvalue-functions 
\begin{figure}[t]
\centering
 \includegraphics[width=9cm]{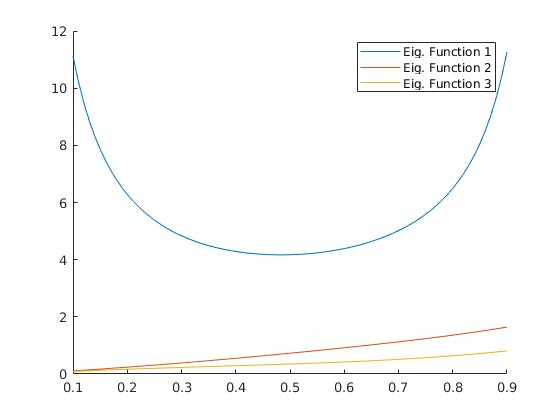}
  \caption{Behavior of the eigenvalues of $\mathrm{A}_2(t)$.}\label{FigExorig}
\end{figure} 
For this matrix $\Bs^{\uparrow}_{\max}=(1/2,1)$ and $\Bs^{\downarrow}_{\max}=\emptyset$. In other
words, for $\Ar_2(t)$ the sufficient conditions of Theorem \ref{lg} give no information about the intervals where
$\la_{\max}(t)$ decreases. Nevertheless, also for this matrix $\Bs^{\uparrow}_{\max}$ provides better information
than $\As^{\uparrow}_{\max}$ of (\ref{eq1}) (see Corollary \ref{coro.impcond}). In fact, $\As^{\uparrow}_{\max}=
\As^{\downarrow}_{\max}=\emptyset$ for $\Ar_2(t)$.\qed
}\end{eje}
}
\item[(ii)] {\rm It is easily seen that
\be\label{eq.eprimeb}
(\sqrt{e_j})'\geq 0 \quad \Leftrightarrow\quad \frac{\left(a_{j-1}b_j\right)'}{a_{j-1}b_j}\geq 
\frac{\left(a_jb_{j+1}\right)'}{a_jb_{j+1}}.
\ee
This implies that $(a_jb_{j+1})'\geq 0$ and $e'_j\geq 0$ is a stronger condition than
$(a_{j-1}b_j)'\geq 0$ and $(a_jb_{j+1})'\geq 0$ in the sense that, for each $j$,
\[
\{t\in \Is\mid (a_jb_{j+1})'\geq 0 , e'_j\geq 0\}\subset \{t\in\Is \mid (a_{j-1}b_j)'\geq 0, (a_jb_{j+1})'\geq 0\}.
\]
This is why $ e'_j\geq 0$ is replaced by $(a_{j-1}b_{j})'\geq 0$ and $(a_jb_{j+1})'\geq 0$ in the first  subset
defining $\Bs^{\uparrow}_{j,\max}$ in (\ref{eq.bjmaxu}).
}
\item[(iii)]{\rm A different set of sufficient conditions can be obtained if $\qr^\mathsf{T} \Sr'\qr$ is written as a 
``sum of squares''.  One can show that  for the eigenvalue-function $\la_k(t)$ of the birth and death matrix
of (\ref{eq.defA}), 
\be\label{eq.sumsqo}
\qr^\mathsf{T} \Sr' \qr=\sum_{j=0}^{n-1}\Big((a'_j+b'_j)+\dps\frac{\Pi'_j}{\Pi_j}(\la_k-(a_j+b_j))\Big) q_j^2+(a'_n+b'_n)q_n^2,
\ee
where $\Pi_j$ can be defined recursively as follows:
\be\label{eq.defPij}
\Pi_n=1,\qquad \Pi_j\Pi_{j+1}= a_jb_{j+1},\quad j=0,1,\ldots, n-1.
\ee
Explicit expressions for these continuous (in $\Is$) functions can be provided. In particular, 
it can be shown using (\ref {eq.defPij}) that for $k=0,1,\ldots, n-1$, if $n-k$ is odd then
\[
\dps\frac{\Pi'_{k}}{\Pi_{k}}=\sum_{j=0}^{\frac{n-k-1}{2}}  \frac{(a_{k+2j}b_{k+2j+1})'}{a_{k+2j}b_{k+2j+1}}-
\sum_{j=1}^{\frac{n-k-1}{2}}  \frac{(a_{k+2j-1}b_{k+2j})'}{a_{k+2j-1}b_{k+2j}}.
\]
And if $n-k$ is even then
\[
\frac{\Pi'_{k}}{\Pi_{k}}=\sum_{j=0}^{\frac{n-k-2}{2}}  \frac{(a_{k+2j}b_{k+2j+1})'}{a_{k+2j}b_{k+2j+1}}-
\sum_{j=0}^{\frac{n-k-2}{2}}  \frac{(a_{k+2j+1}b_{k+2j+2})'}{a_{k+2j-1}b_{k+2j}}.
\]
It is plain that if, for $t\in \Is$, $(a'_j(t)+b'_j(t))>0$ for $j=0,1,\ldots, n$ and $\Pi'_j(t)>0$ for  $j=0,1,\ldots, n-1$ then
$\qr^\mathsf{T}\Sr'\qr>0$. Condition $$\dps\frac{\Pi'_j(t)}{\Pi_j(t)}>0$$ can be seen as a generalization of (\ref{eq.eprimeb}).
}\qed
\end{itemize} 
\end{obs}

It is plain from item (iii) of Observation \ref{obs.lmax} that different ways of writing $\qr^\mathsf{T}\Sr'\qr$ may provide
distinct subsets of $\Is$ where $\la_{\max}(t)$ increases or decreases. An interesting expression of $\qr^\mathsf{T}\Sr'\qr$
that will be of interest for us can be obtained by manipulating (\ref {eq.qtS'q}) a little bit.  In fact, on the
one hand,
\[
2\left(\sqrt{a_jb_{j+1}}\right)'=\frac{(a_jb_{j+1})'}{\sqrt{a_jb_{j+1}}}=a'_j\sqrt{\frac{b_{j+1}}{a_j}}+
b'_{j+1}\sqrt{\frac{a_j}{b_{j+1}}}.
\]
Substituting this expression in (\ref {eq.qtS'q}):
\[
\qr^\mathsf{T}\Sr'\qr=\sum_{j=0}^{n-1} \left((a'_j+b'_j)q_j^2+\left(a'_j\sqrt{\frac{b_{j+1}}{a_j}}+
b'_{j+1}\sqrt{\frac{a_j}{b_{j+1}}}\right) q_{j}q_{j+1}\right)= \sum\limits_{j=0}^{n} s_j q_j
\]
where
\begin{align*}
s_0&=(a'_0+b'_0)q_0+a'_0\dps\sqrt{\frac{b_{1}}{a_0}}q_1,\\
s_j&=(a'_j+b'_j)q_j+b'_{j}\dps\sqrt{\frac{a_{j-1}}{b_{j}}} q_{j-1}+a'_j\sqrt{\frac{b_{j+1}}{a_j}}q_{j+1},\quad j=1,\ldots, n-1,\\
s_n&=(a'_n+b'_n)q_n+b'_n\dps\sqrt{\frac{a_{n-1}}{b_n}}q_{n-1}.
\end{align*}
Recalling that $q_{n+1}(\la_k)=0$ (Proposition \ref{prop.polJac}), we can write 
$$s_n=(a'_n+b'_n)q_n+b'_n\sqrt{\frac{a_{n-1}}{b_n}}q_{n-1}+
a'_n\sqrt{\frac{b_{n+1}}{a_n}}q_{n+1},$$ where $b_{n+1}(t)$ is any positive continuous with continuous first derivative
function in $\Is$. Thus,
\[
s_j=(a'_j+b'_j)q_j+b'_{j}\dps\sqrt{\frac{a_{j-1}}{b_{j}}} q_{j-1}+a'_j\dps\sqrt{\frac{b_{j+1}}{a_j}}q_{j+1},\quad j=1,\ldots, n.
\]
Using again (\ref{eq.defqlk}) to compute $q_{j+1}$ in terms of $q_j$ and $q_{j-1}$ for the eigenvalue-function
$\la_k$:
\begin{align*}
s_0&=\left((a'_0+b'_0)+a'_0\dps\sqrt{\frac{b_{1}}{a_0}}\dps\frac{1}{\sqrt{a_0b_1}}(\la_{k}-(a_0+b_0))\right)q_0\\
&=\dps\frac{1}{a_0}\left(a'_0\la_k +a_0b'_0-a'_0b_0\right)q_0,\\
s_j&=(a'_j+b'_j)q_j+b'_{j}\dps\sqrt{\frac{a_{j-1}}{b_{j}}} q_{j-1}+\frac{a'_j}{a_j} 
\left((\la_k-(a_j+b_j))q_j-\dps\sqrt{a_{j-1}b_j}q_{j-1}\right)\\
&=\dps\frac{1}{a_j}\left((a'_j\la_k+a_jb'_j-a'_jb_j)q_j+\dps\sqrt{\frac{a_{j-1}}{b_j}}(a_jb'_j-a'_jb_j) q_{j-1}\right),
\quad j=1,\ldots, n.
\end{align*}
Therefore
\begin{align}
\label{eq.qtSpqnew}\qr^\mathsf{T}\Sr'\qr&=\dps \frac{1}{a_0}\left(a'_0\la_k +a_0b'_0-a'_0b_0\right)q_0^2\\
\nonumber &\quad +\sum\limits_{j=1}^{n}\dps\frac{1}{a_j}\left((a'_j\la_k+a_jb'_j-a'_jb_j)q_j^2+
\dps\sqrt{\frac{a_{j-1}}{b_j}}(a_jb'_j-a'_jb_j) q_{j-1}q_j\right).
\end{align}
Bearing in mind this formula of $\qr^\mathsf{T}\Sr'\qr$ when applied to $\la_{\max}$, the following theorem can be proved
using similar techniques to those  of Theorem \ref{lg}.

\begin{theorem}\label{lg1}
Let $\Ar(t)$ be the birth and death matrix of (\ref{eq.defA}) and let $m_2(t)$ be the function defined in
(\ref{eq.sigrom}).  Define the following subsets of $\Is$:
\begin{align}
\wt\Bs^{\uparrow}_{0,\max}&=
\nonumber \Big\{t\in\Is\mid a'_0(t)m_2(t)+a_0(t)b'_0(t)-a'_0(t)b_0(t)\geq 0\Big\},\\
\label{eq.bjmaxu1}\wt\Bs^{\uparrow}_{j,\max}&=
\Big\{t\in\Is\mid a_{j}(t)b'_{j}(t)-a'_j(t)b_j(t)\geq 0\\ 
\nonumber&\quad \text{ \rm and } 
a'_j(t)m_2(t)+a_j(t)b'_j(t)-a'_j(t)b_j(t)\geq 0\Big\},\quad j=1,\ldots, n,\\
\label{eq.Nmaxu1}\wt\Ns&=\Big\{t\in\Is\mid a'_0(t)m_2(t)+a_0(t)b'_0(t)-a'_0(t)b_0(t)\neq 0 \text{ \rm or }\\
\nonumber&\quad \; \exists j\in\{1,\ldots, n\} \text{ \rm such that } a_j(t)b'_j(t)-a'_j(t)b_j(t)\neq 0\\
\nonumber&\quad \text{ \rm or }a'_j(t)m_2(t)+a_j(t)b'_j(t)-a'_j(t)b_j(t)\neq 0 \Big\}.
\end{align}
The sets $\wt\Bs^{\downarrow}_{j,\max}$ are defined analogously  by exchanging the roles of $\geq$ 
and $\leq$.
Let $\wt\Bs^{\uparrow}_{\max}=\left(\bigcap_{j=0}^n \wt\Bs^{\uparrow}_{j,\max}\right)\bigcap \wt\Ns$ and 
$\ \wt\Bs^{\downarrow}_{\max}=\left(\bigcap_{j=0}^n \wt\Bs^{\downarrow}_{j,\max}\right)\bigcap \wt\Ns$.
Then $\la_{\max}(\Ar,t)$ is a strictly increasing (resp., strictly decreasing) function of $t$ in each one of the
non-degenerate subintervals of $\wt\Bs^{\uparrow}_{\max}$
(resp., $\wt\Bs^{\downarrow}_{\max}$).
\end{theorem}

Notice that if, for $j\in\{0,1,\ldots,n\}$, $a_j(t)b'_j(t)-a'_j(t)b_j(t)\geq 0$ and $a'_j(t)\geq 0$ then 
$t\in\wt\Bs^{\uparrow}_{j,\max}$. However, it may happen $a'_j(t)< 0$ and still $a'_j(t)\la_{\max}(t)+
a_j(t)b'_j(t)-a'_j(t)b_j(t) \geq 0$. It is enough to require $a_j(t)b'_j(t)-a'_j(t)b_j(t)$  to be as big
as $-m_2(t)a'_j(t)>0$.

\begin{obs}{\rm
A simple computation shows that, for matrix $\Ar_1(t)$ of Example \ref{Meixner},
\[
\sigma(t)=\left\{\begin{array}{l}
\dps\frac{2}{t}+\frac{1}{1-t},\quad 0\leq t\leq \frac{1}{2}\\[7pt]
\dps\frac{1}{t}+\frac{2}{1-t},\quad \frac{1}{2} \leq t\leq 1,
\end{array}\right.
\]
and $\rho(t)=1/t+1/(1-t)+1$. Then $m_2(t)=\rho(t)$, $\wt\Bs_{0,\max}^{\uparrow}=[0.554958,1]$,
$\wt\Bs_{1,\max}^{\uparrow}=[0,1]$ and $\wt\Bs_{2,\max}^{\uparrow}=[0.445042,1]$. Hence
$\wt\Bs_{\max}^{\uparrow}=[0.554958,1]$. Also, $\wt\Bs_{\max}^{\downarrow}=\emptyset$
because $a_1(t)b'_1(t)-a'_1(t)b_1(t)=1>0$ for all $t\in[0,1]$.  For this matrix , $\wt\Bs_{\max}^{\uparrow}$ and 
$\wt\Bs_{\max}^{\downarrow}$ provide less information than $\Bs_{\max}^{\uparrow}$ and 
$\Bs_{\max}^{\downarrow}$.\qed
}
\end{obs}

For the smallest eigenvalue of $\mathrm{A}(t)$, the sign patterns of the entries of the 
corresponding eigenvector are  quite controllable (see  (\ref{eq.sigchan})). This fact and having an explicit
expression for $\la_{\min}^\prime(\Ar,t)$ in terms of its eigenvectors (cf. (\ref{eq.deriveig})) allow us to
study the monotonicity of the smallest eigenvalue-function in $\Is$. The subsets where it increases or decreases
look very much like the ones in Theorem \ref{lg}.

\begin{theorem}\label{lg2}
Let $\Ar(t)$ be the birth and death matrix of (\ref{eq.defA}) and let $e_j(t)$,
$\ell_j(t)$, $j=0,1,\ldots, n-1$, be the functions of (\ref{eq.defefgl}). For each $t\in\Is$ set
\be\label{eq.mu}
\dps{\mu(t)=\min_{0\leq j\leq n}\{a_j(t)+b_j(t)\}.}
\ee 
Let 
\begin{align*}
h_n(t)&=a'_n(t)+b'_n(t)-\ell_{n-1}(t)(a_n(t)+b_n(t)),\\
l_n(t)&=a'_n(t)+b'_n(t)+\ell_{n-1}(t)(\mu(t)-a_n(t)-b_n(t)),
\end{align*}
and, for $j=0,1,\ldots, n-1$,
\begin{align*}
h_j(t)&=a'_j(t)+b'_j(t)-\ell_{j}(t)(a_j(t)+b_j(t)),\\
l_j(t)&=a'_j(t)+b'_j(t)+\ell_{j}(t)(\mu(t)-a_j(t)-b_j(t)). 
\end{align*}
Define the following subsets of $\Is$:
\begin{align}
\label{eq.bjminu}\Bs^{\uparrow}_{j,\min}&=\Big\{t\in\Is\mid \Big((a_{j-1}(t)b_{j}(t))'\leq 0 \text{ \rm and }
 (a_{j}(t)b_{j+1}(t))'\leq 0 \text{ \rm and } a'_j(t)+b'_j(t) \geq 0 \Big)\\ 
\nonumber&\quad \left. \text{ \rm or } \Big((a_{j}(t)b_{j+1}(t))'\geq 0 \text{ \rm and } h_j(t)\geq 0 \text{ \rm and } e'_j(t)\leq 0\Big)\right.\\
\nonumber&\quad \left. \text{ \rm or } \Big(a'_{j}(t)+b'_{j}(t)\leq 0 \text{ \rm and } l_j(t)\geq 0 \text{ \rm and } e'_j(t)\leq 0\Big)\right\}, \quad j=0,1,\ldots, n-1,\\
\nonumber\Bs^{\uparrow}_{n,\min}&=\Big\{t\in\Is\mid  \Big((a_{n-1}(t)b_{n}(t))'\leq 0 \text{ \rm and }
a'_n(t)+b'_n(t) \geq 0\Big)\\
\nonumber&\quad \text{ \rm or } \Big((a_{n-1}(t)b_{n}(t))'\geq 0 \text{ \rm and } h_n(t)\geq 0
\Big) \text{ \rm or } \Big(a'_{n}(t)+b'_{n}(t)\leq 0 \text{ \rm and } l_n(t)\geq 0
\Big)\Big\},
\end{align}
and let $\Ns$ be the set of (\ref{eq.Nmaxu}). The sets $\Bs^{\downarrow}_{j,\min}$ are defined analogously  by exchanging the roles of $>$ and $<$ on the one hand
 and $\geq$ and $\leq$ on the other hand.
Let $\Bs^{\uparrow}_{\min}=\left(\bigcap_{j=0}^n \Bs^{\uparrow}_{j,\min}\right)\bigcap \Ns$ and 
$\ \Bs^{\downarrow}_{\min}=\left(\bigcap_{j=0}^n \Bs^{\downarrow}_{j,\min}\right)\bigcap \Ns$.
Then $\la_{\min}(\Ar,t)$ is a strictly increasing (resp., strictly decreasing) function of $t$ in each one of the
non-degenerate subintervals of $\mathsf{B}^{\uparrow}_{\min}$
(resp., $\mathsf{B}^{\downarrow}_{\min}$).
\end{theorem}

\medskip
\begin{proof}
The proof is very similar to that of Theorem \ref{lg}.  First, let $\qr=\qr_0(\la_{\min}(t),t)$ be the
eigenvector-function of (\ref{eq.qk}) for the eigenvalue-function $\la_{\min}=\la_{\min}(\Ar,t)$.
Let $q_j=q_j(\la_{\min}(t);t)$ be the $j$-th coordinate of $\qr$.
According to (\ref{eq.sigchan}), for each $t\in\Is$, the signs of the coordinates of $\qr$ alternate. Since $q_0=1$, we
have $\sgn(q_j)=(-1)^{j}$ and so $q_{j-1}q_j<0$. On the other hand,  for each $t\in\Is$ if $a_h(t)+b_h(t)=\mu(t)$
then $$\la_{\min}(\Ar,t)=\la_{\min}(\Sr,t)=\min_{\|u\|_2=1}u^\mathsf{T} \Sr(t)  u\leq e_{h}^\mathsf{T} \Sr(t) e_{h}=\mu(t).$$ Again,
by using (\ref{eq.interlacing}) and (\ref{eq.defqlk})  it can be seen that  $\la_{\min}(\Ar,t)<\mu(t)$. Also, from
(\ref{eq.poseig}), $\la_{\min}(\Ar,t)>0$. Let
\be\label{eq.qSqnmin}
(\qr^\mathsf{T}\Sr'\qr)_n=\left((a'_n+b'_n)+\ell_{n-1}(\la_{\min}-(a_n+b_n))\right) q_n^2,
\ee
and for $j\in\{0,1,\ldots, n-1\}$, let 
\be\label{eq.qSqjmin}
(\qr^\mathsf{T}\Sr'\qr)_j=\left(\left((a'_j+b'_j)+\ell_j\left(\la_{\min}-(a_j+b_j)\right)\right)q_j^2+
\sqrt{a_jb_{j+1}}\ \left(\sqrt{e_j}\right)'\ q_{j-1}\ q_j\right).
\ee
Bearing in mind that $q_{j-1}q_j<0$ and 
$0<\la_{\min}(\Ar,t)< \mu(t)$, the technique of the proof of Theorem \ref{lg} can be used to show that
$(\qr^\mathsf{T} \Sr' \qr)_j\geq 0$ for all $t\in\Bs^{\uparrow}_{j,\min}$, $j=0,1,\ldots, n$. The theorem follows from
(\ref{eq.qtS'qsimpli}) and the fact that if $t\in\Ns$ then $(\qr^\mathsf{T} \Sr' \qr)_j> 0$  for some $j\in\{0,1,\ldots, n\}$.
\end{proof}

\begin{obs}\label{obs.lmin0}{\rm 
For the matrix $\Ar_1(t)$ of Example \ref{Meixner}, it is easily checked that $\Bs^{\uparrow}_{\min}=(1/2,1)$
and $\ \Bs^{\downarrow}_{\min}=(0,1/2)$ in concordance with what is shown in Figure \ref{FigEx22}. Analysing
these sets for $\Ar_2(t)$ is a little more involved, but taking into account that $\mu(t)=a_0(t)+b_0(t)=t$ it can be
seen that  $\Bs^{\uparrow}_{0,\min}=\Bs^{\uparrow}_{2,\min}=(0,1)$ and $\Bs^{\uparrow}_{1,\min}=(3/5,1)$.
Therefore $\Bs^{\uparrow}_{\min}=(3/5,1)$. On the other hand, $\Bs^{\downarrow}_{0,\min}=\emptyset$ so that
$\Bs^{\downarrow}_{\min}=\emptyset$. All this is consistent with the information about the intervals
where $\la_{\min}(t)$ increases and decreases provided by Figure \ref{FigExorig}.\qed
}\end{obs}

\medskip
The problem of the monotonicity of $\la_{\min}(t)$  when $h_j(t)<0$ or $l_j(t)<0$
is still open. It is reasonable to expect that in these cases condition $e'_j(t)\leq 0$ 
will not be enough and it should be required to be smaller than a negative quantity depending  on $h_j(t)$ or $l_j(t)$, 
respectively.  By following a lead of \cite[Th. 2.2]{I89} we are to show that this is indeed the case. To begin with, let $\{q_j(x;t)\}$ be the (orthogonal) polynomials associated
to the Jacobi matrix $\Sr(t)$ of (\ref{eq.defS}) (see (\ref{eq.delpeivec})):
\begin{align}
\nonumber q_{-1}(x;t)&=0,\quad q_0(x;t)=1,\\
\label{eq.ortpolS} &\sqrt{a_j(t)b_{j+1}(t)} q_{j+1}(t)(x;t)\\
\nonumber&\quad=(x-(a_j(t)+b_j(t)))q_j(x;t)-\sqrt{a_{j-1}(t)b_j(t)} q_{j-1}(t),\quad j=1,2,\ldots . 
\end{align}
It is not difficult to see by induction that
\be\label{eq.qjzero}
q_j(0;t)=(-1)^j\,\frac{\mathlarger{\sum}\limits_{k=0}^j \dps\prod\limits_{i=k}^{j-1}a_i(t)\dps\prod\limits_{i=0}^{k-1} b_i(t)}
{\sqrt{\dps\prod\limits_{k=0}^{j-1} a_k(t)\dps\prod\limits_{k=1}^j b_k(t) }}
\ee
where we are agreeing that $\dps\prod_{i=0}^{-1}=\dps\prod_{i=j}^{j-1}=1$. Let us denote $\chi_{-1}(t)=0$,
 $\chi_0(t)=1$ and for $j=1,\ldots, n$
\be\label{eq.defji}
\chi_j(t)=\dps\frac{\mathlarger{\sum}\limits_{k=0}^j b_0(t)b_1(t)\cdots b_{k-1}(t)a_k(t)a_{k+1}(t)\cdots a_{j-1}(t) }
{\dps\sqrt{a_0(t)a_1(t)\cdots a_{j-1}(t)b_1(t)b_2(t)\cdots b_j(t)}}.
\ee
Then $q_0(0;t)=1$ and $q_j(0;t)=(-1)^j \chi_j(t)$ for $j=1,\ldots, n$.

With the notation of Theorem \ref{lg2}, for $j=1,\ldots, n-1$, we define
\begin{align}\label{eq.derj}
r_j(x;t)&=\sqrt{a_j(t)b_{j+1}(t)}\left(\sqrt{e_j(t)}\right)'q_{j-1}(x;t)+h_j(t)q_j(x;t),\\
s_j(x;t)&=\sqrt{a_j(t)b_{j+1}(t)}\left(\sqrt{e_j(t)}\right)'q_{j-1}(x;t)+l_j(t)q_j(x;t).
\end{align}

\begin{lemma}\label{lem.signrs}
Let $j\in\{1,\ldots, n-1\}$.
\begin{itemize}
\item[(a)] If for $t\in\Is$ $h_j(t)<0$ and $$u_j(t)=\sqrt{a_j(t)b_{j+1}(t)}\left(\sqrt{e_j(t)}\right)'\chi_{j-1}(t)-
h_j(t)\chi_j(t)<0,$$
then $\sgn (r_j(\la_{min}(t);t)=(-1)^j$.
\item[(b)] If for $t\in\Is$ $l_j(t)<0$ and $$v_j(t)=\sqrt{a_j(t)b_{j+1}(t)}\left(\sqrt{e_j(t)}\right)'\chi_{j-1}(t)-
l_j(t)\chi_j(t)<0,$$
then $\sgn (s_j(\la_{min}(t);t)=(-1)^j$.
\end{itemize}
\end{lemma}

\begin{proof}
We will prove item (a); the proof of item (b) is similar. We take any $t\in\Is$ but for notational simplicity
we will omit the dependence on $t$ of all functions. We will also assume that $j$ is any integer between
$1$ and $n-1$. Let us compute $r_j(0)$ (for the chosen arbitrary $t$):
\begin{align*}
r_j(0)&= \sqrt{a_jb_{j+1}}\left(\sqrt{e_j}\right)'q_{j-1}(0)+h_jq_j(0)\\
&= (-1)^{j-1}\sqrt{a_jb_{j+1}}\left(\sqrt{e_j}\right)'\chi_{j-1}+(-1)^j h_j \chi_j\\
&=(-1)^{j-1}\left(\sqrt{a_jb_{j+1}}\left(\sqrt{e_j}\right)'\chi_{j-1}- h_j \chi_j\right)\\
&=(-1)^{j-1}u_j.
\end{align*}
Since $u_j<0$ in $\Is$, $\sgn(r_j(0))=(-1)^j$. Now
\[
\frac{1}{h_j}r_j(x)=q_j(x)+\sqrt{a_jb_{j+1}}\frac{\big(\sqrt{e_j}\big)'}{h_j}q_{j-1}(x).
\]
It follows from $h_j<0$ and $u_j<0$ that
\[
\sqrt{a_jb_{j+1}}\frac{\big(\sqrt{e_j}\big)'}{h_j}\chi_{j-1}>\chi_j,
\]
and since $a_j$, $b_{j+1}$, $\chi_{j-1}$, $\chi_j$ are all positive, $\left(\sqrt{e_j}\right)'/h_j>0$. Let $\la_{1k}<\la_{2k}<\cdots <\la_{kk}$ be the roots of $q_k(x)$. Then (see (\ref{eq.interlacing})
$\la_{ij}<\la_{ij-1}<\la_{i-1 j}$ and so (see \cite[p. 40]{C78}) the roots of $r_j(x)$,
$\alpha_1<\alpha_2<\cdots <\alpha_n$ say,  are real and
\be\label{eq.rointzq}
\la_{k-1j-1}<\alpha_k<\la_{kj},\quad k=1,2,\ldots, n,\quad (\la_{0j-1}=-\infty).
\ee
In particular $\alpha_1<\la_{1j}$ and, as $\la_{\min}$ is the smallest root of $q_{n+1}(\la)$
(Proposition \ref{prop.polJac}), it follows from the interlacing inequalities (\ref{eq.interlacing}) that
\[
0<\la_{\min}<\la_{1j}<\la_{1j-1}.
\]
We claim that $0\geq \alpha_1$. In order to see this, let us show first that $\sgn(r_j(\la_{1j}))=(-1)^j$.
In fact, $r_j(\la_{1j})=\sqrt{a_jb_{j+1}}\left(\sqrt{e_j}\right)'q_{j-1}(\la_{1j})$ because $q_j(\la_{1j})=0$. But
$\la_{1j-1}$ is the smallest root of $q_{j-1}(x)$, $0<\la_{\min}<\la_{1j}<\la_{1j-1}$ and it follows
from (\ref{eq.qjzero}) that $\sgn(q_{j-1}(0))=(-1)^{j-1}$. Hence,
$\sgn(q_{j-1}(\la_{1j}))=\sgn(q_{j-1}(\la_{\min}))=\sgn(q_{j-1}(0))=(-1)^{j-1}$. Now $\left(\sqrt{e_j}\right)'<0$
because $\left(\sqrt{e_j}\right)'/h_j>0$ and $h_j<0$. Therefore $\sgn(r_j(\la_{1j}))=(-1)^j$. On the other hand we have already seen that $\sgn(r_j(0))=(-1)^j$ and $\alpha_1$ is the only root of $r_j(x)$ smaller
than $\la_{1j}$. Since the sign of $r_j$ in $\la_{1j}$ and $0$ coincide, we must have $\alpha_1<0$ as claimed.
Given that $r_j(x)$ has the same sign in the whole interval $(0,\la_{1j})$ and $0<\la_{\min}<\la_{1j}$ we
get $r_j(\la_{\min})= r_j(\la_{1j})=r_j(0)=(-1)^j$, as desired.
\end{proof}

\begin{theorem}\label{0}
 Assume that the conditions and notation of Theorem \ref{lg2} and Lemma \ref{lem.signrs} hold. 
 Define  $\Es^{\uparrow}_{j,\min}=\Bs^{\uparrow}_{j,\min}$ for $j=0,n$ and the following subsets of $\mathsf{I}:$
 \begin{align}
 \label{eq.ejminu}\Es^{\uparrow}_{j,\min}&=\Big\{t\in\Is\mid \Big((a_{j-1}(t)b_{j}(t))'> 0 \text{ \rm and } h_j(t)< 0
\text{ \rm and } \\
\nonumber&\qquad \left(\sqrt{e_j(t)}\right)'< \dps\frac{1}{\sqrt{a_j(t)b_{j+1}(t)}}\dps\frac{\chi_j(t)}{\chi_{j-1}(t)}h_j(t)\Big)
\text{ \rm or } \Big((a_{j-1}(t)b_{j}(t))'< 0  \text{ \rm and }\\
\nonumber&\qquad l_j(t)< 0 \text{ \rm and } \left(\dps\sqrt{ e_j(t)}\right)'<\dps\frac{1}{\sqrt{a_j(t)b_{j+1}(t)}}\dps\frac{\chi_j(t)}{\chi_{j-1}(t)}l_j(t)\Big)\Big\}, \quad j=1,\ldots, n-1.
 \end{align}
The set $\Es^{\downarrow}_{j,\min}$ is defined analogously  by exchanging the roles of $>$ and $<$ on the one hand
 and $\geq$ and $\leq$ on the other hand.
 Let $\Es^{\uparrow}_{\min}=\left(\bigcap_{j=0}^{n} \Es^{\uparrow}_{j,\min}\right)$ and 
$\ \Es^{\downarrow}_{\min}=\left(\bigcap_{j=0}^n \Es^{\downarrow}_{j,\min}\right)$.
Then $\la_{\min}(\Ar,t)$ is a strictly increasing (resp., strictly decreasing) function of $t$ in each one of the
non-degenerate subintervals of $\Es^{\uparrow}_{\min}$
(resp., $\Es^{\downarrow}_{\min}$).
 \end{theorem}
 \begin{proof}
 We use the expression of $\qr^\mathsf{T}\Sr'\qr$ of (\ref{eq.qtS'qsimpli}).
 Let $(\qr^\mathsf{T}\Sr'\qr)_n$ and $(\qr^\mathsf{T}\Sr'\qr)_j$, $j=0,1,\ldots, n$, be the functions defined in (\ref {eq.qSqnmin}) and
 (\ref {eq.qSqjmin}) respectively.  It was proven in Theorem \ref{lg2} that if $t\in\Es^{\uparrow}_{j,\min}$, $j=0,n$,
 then $(\qr^\mathsf{T}\Sr'\qr)_0\geq 0$ and $(\qr^\mathsf{T}\Sr'\qr)_n\geq 0$. Let $j$ be any integer between $1$ and $n-1$ and assume
 that $t\in \Es^{\uparrow}_{j,\min}$. Let $u_j(t)$ and $v_j(t)$ be the functions defined in the statement of Lemma
 \ref{lem.signrs}.
 \begin{itemize}
 \item If $(a_{j-1}(t)b_{j}(t))'> 0$ and $h_j(t)< 0$ and 
 $$\sqrt{e_j(t)}\big)'< \frac{1}{\sqrt{a_j(t)b_{j+1}(t)}}\frac{\chi_j(t)}{\chi_{j-1}(t)}h_j(t)$$ then $\ell_j(t)>0$ and $u_j(t)<0$.
 Thus, removing the dependence on $t$,
 \begin{align*}
 (\qr^\mathsf{T}\Sr'\qr)_j&= (a'_j+b'_j+\ell_j(\la_{\min}-(a_j+b_j)))q_j^2+\sqrt{a_jb_{j+1}}\left(\sqrt{e_j}\right)' q_{j-1}q_j\\
 &> (a'_j+b'_j- \ell_j(a_j+b_j))q_j^2+\sqrt{a_jb_{j+1}}\left(\sqrt{e_j}\right)' q_{j-1}q_j\\
 &= (h_j q_j+\sqrt{a_jb_{j+1}}\left(\sqrt{e_j}\right)' q_{j-1})q_j\\
 &= r_j(\la_{\min})q_j,
 \end{align*}
 where, in the first inequality, we have used that $\ell_j>0$ and $\la_{\min}>0$. Now, $\sgn(q_j)=(-1)^j$ and
 by Lemma \ref{lem.signrs}, $\sgn(r_j(\la_{\min}))=(-1)^j$. As a consequence, $(\qr^\mathsf{T}\Sr'\qr)_j>0$.
 
 \item If $(a_{j-1}(t)b_{j}(t))'<0$ and $l_j(t)< 0$ and
 $$\sqrt{e_j(t)}\big)'< \frac{1}{\sqrt{a_j(t)b_{j+1}(t)}}\frac{\chi_j(t)}{\chi_{j-1}(t)}l_j(t)$$ then $\ell_j(t)<0$ and $v_j(t)<0$.
 As in the previous case,
 \begin{align*}
 (\qr^\mathsf{T}\Sr'\qr)_j&= (a'_j+b'_j+\ell_j(\la_{\min}-(a_j+b_j)))q_j^2+\sqrt{a_jb_{j+1}}\left(\sqrt{e_j}\right)' q_{j-1}q_j\\
 &> (a'_j+b'_j+ \ell_j(\mu- (a_j+b_j)))q_j^2+\sqrt{a_jb_{j+1}}\left(\sqrt{e_j}\right)' q_{j-1}q_j\\
 &= (l_j q_j+\sqrt{a_jb_{j+1}}\left(\sqrt{e_j}\right)' q_{j-1})q_j\\
 &= r_j(\la_{\min})q_j,
 \end{align*}
 where the first inequality follows from $\ell_j<0$ and $\la_{\min}<\mu$. Since $\sgn(q_j)=(-1)^j$ and
 $\sgn(r_j(\la_{\min}))=(-1)^j$, $(\qr^\mathsf{T}\Sr'\qr)_j>0$ as desired.
 \end{itemize}
\end{proof}

\begin{obs}\label{obs.lmin}{\rm
The case $b_0(t)=0$ deserves special attention.  Notice that for $j=1,2,\ldots$,
\be\label{eq.chijjm1}
\chi_j=\sqrt{\frac{a_{j-1}}{b_j}}\chi_{j-1}+\frac{b_0}{b_jd_j}.
\ee
where $d_0(t)$, $d_1(t)$, \ldots,  are the functions of (\ref{eq.defdj}).
Thus, when $b_0=0$ the conditions 
$$
\big(\sqrt{e_j}\big)'< \frac{1}{\sqrt{a_jb_{j+1}}}\frac{\chi_j}{\chi_{j-1}}h_j, \quad \big(\sqrt{e_j}\big)'< \frac{1}{\sqrt{a_jb_{j+1}}}\frac{\chi_j}{\chi_{j-1}}l_j,$$
defining the set $\Es^{\uparrow}_{j,\min}$ reduce to the easier to compute conditions
$$
\big(\sqrt{e_j}\big)'< \frac{1}{\sqrt{a_jb_{j+1}}}\sqrt{\frac{a_{j-1}}{b_j}}h_j, \quad \big(\sqrt{e_j}\big)'< \frac{1}{\sqrt{a_jb_{j+1}}}\sqrt{\frac{a_{j-1}}{b_j}}l_j,$$ respectively.

On the other hand, if $b_0\neq 0$ then the sets $\As^{\uparrow}_{\min}$ and $\As^{\downarrow}_{\min}$
defined in (\ref{eq2}) are empty. This is the case, for example, of matrices $\Ar_1(t)$ and $\Ar_2(t)$ of
Examples \ref{Meixner} and \ref{eje.A2}, respectively. 
It is worth-noticing in this respect that the condition $a'_j(t)b_j(t)-a_j(t)b'_j(t)>0$
defining the set $\As^{\uparrow}_{\min}$ is closely related to the expression of $\qr^\mathsf{T} \Sr'\qr$ in
(\ref{eq.qtSpqnew}) for $\la_{\min}(t)$. In fact, if we define for $j=1,2,\ldots, n$
\[
z_j(t;x)=q_j(t;x)+\sqrt{\frac{a_{j-1}(t)}{b_j(t)}}q_{j-1}(t;x),
\]
then we get in (\ref{eq.qtSpqnew})
\[
(a'_j\la_{\min}+a_jb'_j-a'_jb_j)q_j^2+\sqrt{\frac{a_{j-1}}{b_j}}(a_jb'_j-a'_jb_j)q_{j-1}q_j=
(a'_j\la_{\min}+ (a_jb'_j-a'_jb_j) z_j(\la_{\min}))q_j.
\]
As in (\ref{eq.rointzq}), the roots of $z_j(x)$ are real and if they are $\beta_1<\beta_2<\cdots<\beta_{j}$ then
$\beta_1<\la_{1j}<\la_{1j-1}<\beta_2$. It follows from (\ref{eq.chijjm1}) that when $b_0=0$
\[
z_j(0)=q_j(0)+\sqrt{\frac{a_{j-1}}{b_j}} q_{j-1}(0)=(-1)^j \chi_j+\sqrt{\frac{a_{j-1}}{b_j}}  (-1)^{j-1}\chi_{j-1}=0.
\]
In other words, $\beta_1=0$ and $z_j(x)$ does not change sign in the interval $(0,\la_{1j}]$. But
$$\sgn(z_j(\la_{1j}))=\sgn(q_{j-1}(\la_{1j})) =(-1)^{j-1}.$$ Bearing in mind that $\sgn(q_j(\la_{\min}))=(-1)^j$,
if $a'_j>0$ and $(a_jb'_j-a'_jb_j)<0$ then 
$$(a'_j\la_{\min}+a_jb'_j-a'_jb_j)q_j^2+\sqrt{\frac{a_{j-1}}{b_j}}(a_jb'_j-a'_jb_j)q_{j-1}q_j>0.$$ On the other hand
a sufficient condition for $a'_0\la_{\min}+a_0b'_0- a'_0b_0$ is $a'_0>0$ because we are assuming that $b_0=0$.
These are Ismail's conditions defining $\As^{\uparrow}_{\min}$ in \eqref{eq2}. If $b_0\neq 0$ then $z_j(0)=(-1)^j \frac{b_0}{b_jd_j}$ and so $z_j(x)$ does change sign in $(0,\la_{1j})$.
Whether there are Ismail-like conditions applying in this case remains an open problem.\qed
}\end{obs}

\section{Intermediate eigenvalues of birth and death matrices}\label{sec.QueQ}
As mentioned in the introduction section, Magagna addressed the problem of the monotonicity of the eigenvalues of 
homogeneous (i.e.;  time-independent) birth and death matrices in his Ph. D. Thesis of 1965  and \cite{HM70}.
An immediate consequence of his main result (see Theorem \ref{thm.Magagna}) is that if,  for $t\in\Is$,
$a_j(t)=rb_{j+1}(t)$ for $j=0,1,\ldots, n-1$ or $a_j(t)=rb_j(t)$ for $j=1,2,\ldots, n-1$, where $r>0$ is a positive
real number, then the nonzero eigenvalues of $\Ar(t)$, with $b_0(t)=a_n(t)=0$, strictly increase at $t$.
As a result, if we define
\begin{align}
\label{magagna1}\As^{\uparrow}_0&=\Big\{t \in \mathsf{I} \, | \, a'_j(t)> 0 \text{ \rm  and } a'_j(t)b_j(t) = a_j(t)b_j'(t),
\; j=0,1,\ldots, n\Big\},\\
\label{magagna2} \As^{\uparrow}_1&=\Big\{t \in \mathsf{I} \, | \, a'_0(t)> 0  \, \text{ \rm and }\, \big(a'_j(t)> 0 \text{ \rm and } 
a'_{j-1}(t) b_j(t)=a_{j-1}(t) b_j'(t) \big), \; j=1,\ldots, n\Big\},
\end{align}
then the eigenvalues of $\mathrm{A}(t)$ are strictly increasing functions of $t$ in each of the non-degenerate
subintervals of $\As^{\uparrow}_0 \cup \As^{\uparrow}_1$. Actually there is no need to appeal to Magagna and Horne's result in order to prove this property. It is an easy
consequence of our previous developments. In fact, if $t\in  \As^{\uparrow}_0 $ then $$\qr^\mathsf{T} \Sr' \qr=
\dps\sum_{j=1}^n \frac{a'_j(t)}{a_j(t)} \la_k q_j^2$$ in \eqref{eq.qtSpqnew}. Hence, $\la'_k(\Ar,t)=\la'_k(\Sr,t)>0$.
On the other hand, if $t\in \As^{\uparrow}_1$ then one can see after some computations that
$\left(\left(\sqrt{a_j(t)b_{j+1}(t)}\right)'\right)^2= a'_j(t)b'_{j+1}(t)$. Thus, if $\Sr(t)$ is the matrix of \eqref{eq.defS}
then
\[
\Sr'(t)=\begin{pmatrix}
a'_0(t)+b'_0(t) &\dps \sqrt{a'_0(t) b'_1(t)(t)} & &\\
\dps \sqrt{a'_0(t) b'_1(t)} &   \dps a'_1(t)+b'_1(t) &  \ddots & \\%
& \ddots & \ddots &   \dps \sqrt{a'_{n-1}(t) b'_n(t)} \\
& &  \dps \sqrt{a'_{n-1}(t) b'_n(t)} & a'_n(t)+b'_{n}(t)
\end{pmatrix}.
\]
Defining $\wt{\Dr}(t)=\diag(\wt{d}_0(t) \wt{d}_1(t),\ldots, \wt{d}_n(t))$ with
\[
\wt{d}_0(t)=1, \quad \wt{d}_j(t)\sqrt{\frac{a'_0(t) \cdots a'_{j-1}(t)}{b'_1(t)\cdots b'_j(t)}},\quad j=1,\ldots, n,
\]
we get $\wt{\Dr}(t)^{-1}\Sr'(t)\wt{\Dr}(t)=\Ar'(t)$. But, for $t \in\As^{\uparrow}_1$, $\Ar'(t)$ is a birth and death
matrix and so its eigenvalues are all positive. This means that for $t\in\As^{\uparrow}_1$, $\Sr'(t)$ is symmetric
and positive definite and so for the eigenvector $\qr_k(t)$ of \eqref{eq.qk}, $\qr_k(t)^\mathsf{T} \Sr'(t) \qr_k(t)>0$ for
each $t\in \As^{\uparrow}_1$. By \eqref{eq.deriveig}, $\la'_k(\Ar,t)>0$ as claimed. A little more can be said about the relationship between  $\la'_k(\Ar,t)$  and $\la_k(\Ar',t)$ when 
$t\in\As^{\uparrow}_0 \cap \As^{\uparrow}_1$.

\begin{theorem}
Let $\Ar(t)$ be the birth and death matrix of \eqref{eq.defA} and let $\As^{\uparrow}_0$ and $\As^{\uparrow}_1$
be the subsets of  $\, \Is$ defined in \eqref{magagna1} and \eqref{magagna2}. The sets $\As^{\downarrow}_0$
and $\As^{\downarrow}_1$ are defined analogously  by exchanging the roles of $>$ and $<$. Then
\begin{align*}
0<&\la_0(\Ar')=\la'_0(\Ar)<\cdots<\la_n(\Ar')=\la'_n(\Ar)\quad &
\text{ in } \As^{\uparrow}_0 \cap \As^{\uparrow}_1&,\\
0>&\la_0(\Ar')=\la'_0(\Ar)>\cdots>\la_n(\Ar')=\la'_n(\Ar) \quad &
\text{ in } \As^{\downarrow}_0 \cap \As^{\downarrow}_1&. 
\end{align*}
\end{theorem}
\begin{proof}
Assume that $t\in \As^{\uparrow}_0 \cap \As^{\uparrow}_1$. Simple computations show that for $j=0,1,\ldots, n$,
$a_{j-1}(t)a'_j(t)=a'_{j-1}(t)a_j(t)$, $a_{j-2}(t)b'_{j}(t)=a'_{j-2}(t)b_{j}(t)$, and $b_j(t) b'_{j-1}(t)=b'_j(t) b_{j-1}(t)$.  We claim that $\Ar'(t)\Ar(t)=\Ar(t)\Ar'(t)$. Indeed, removing the dependence on $t$,
\begin{align*}
& (\Ar\Ar')_{j,j}=a'_{j-1}b_j+(a_j+b_j)(a'_j+b'_j)+a_j b'_{j+1}=(\Ar'\Ar)_{j,j}\\
&(\Ar\Ar')_{j-1,j}=a'_{j-1}a_{j-1}+a'_{j-1}b_{j-1}+a_{j-1} a'_j+a_{j-1}b'_{j}=(\Ar'\Ar)_{j-1,j}\\
&(\Ar\Ar')_{j,j-1}=a'_{j-1}b_j+b_j b'_{j-1}+a_j b'_j+b_j b'_j=(\Ar'\Ar)_{j,j-1}\\
&\qquad \qquad \ (\Ar\Ar')_{j-2,j}=a_{j-2}a'_{j-1}=(\Ar'\Ar)_{j-2,j}\\
&\qquad \qquad \ \ \ (\Ar\Ar')_{j,j-2}=b_{j}b'_{j-1}=(\Ar'\Ar)_{j,j-2}.
\end{align*}
Since $\mathrm{A}'\mathrm{A}$ is a pentadiagonal matrix, the remaining elements are all zero. The result
follows from Rose's theorem $($cf. \cite[Theorem 2]{R65}$)$.
 \end{proof}
 
\section{Random Walk Matrices}\label{sec.randwalkmat}
As far as random walk matrices, $\Br(t)$ of (\ref{eq.defB}), are concerned the known result 
about $\la_{\max}(\Br)$ only applies when $c_0(t)=1$ (see the set $\Cs^{\uparrow}_{\max}$ of
\eqref{eq3}) .
Without this assumption the conditions defining $\Cs^{\uparrow}_{\max}$ may not be sufficient for
$\la_{\max}(\Br)$ to increase. This is illustrated in Example \ref{ejerw} below. Of course, in some cases, 
for instance when $c_j(t)=1/2 (j+1)/(j+t)$ ($t>-1/2$), 
the problem can be rewriten so that $c_0(t)=1$.

\begin{eje}\label{ejerw}{\rm
Define the $2$-by-$2$ random walk matrix
$$
\Br_1(t)=\begin{pmatrix}
0 & \dps \frac{1}{1+t}\\[7pt]
\dps 1-\frac{1}{1+2t} & 0
\end{pmatrix}  \quad (0<t<1).
$$
Then $c_0(t)=1/(1+t)$ and $c_1(t)=1/(1+2t)$. Obviously, $c'_0(t)<0$ and $c'_1(t)<0$ for all $t\in (0,1)$. However 
$$
\la_{\max}(\Br_1)=\frac{\sqrt{2t}}{(t+1)(2t+1)},
$$
is a strictly increasing function on $(0, 1/\sqrt{2})$ and strictly decreasing on $(1/\sqrt{2}, 1)$. As can be expected 
from Theorem \ref{lg}, this is related to the fact that at $t=1/\sqrt{2}$ the sign of $((1-c_1(t))c_0(t))'$  changes
from positive to negative.\qed
}\end{eje}

\begin{proposition}\label{lrw}
 Let $\Br(t)$ be the random walk matrix of \eqref{eq.defB}. Set $\delta_{j}(t)=(1-c_{j+1}(t))c_{j}(t)$ for
 $j=0, \dots, n-1$ and define the following subset of $\Is$:
 \[
 \Ds^{\uparrow}_{\max}=\left\{t \in \Is\, | \, \forall j\in\{0,\dots,n-1\} \ \ \delta_j'(t)>0\right\}.
 \]
The set $\Ds^{\downarrow}_{\max}$ is defined analogously  by exchanging the roles of $>$ and $<$.
Then $\la_{\max}(\Br,t)$ is a strictly increasing (resp., strictly decreasing) function of $t$ in each one of the
non-degenerate subintervals of $\Ds^{\uparrow}_{\max}$ (resp., $\Ds^{\downarrow}_{\max}$).
\end{proposition}
\begin{proof}
It was shown in Section \ref{sec.prelim} that if  $\Br(t)$ is the  random walk matrix of \eqref{eq.defB} then
$\wh{\Ar}(t)=I_{n+1}+\Br(t)$ is a birth and death matrix with $\wh{a}_i(t)=c_i(t)$ and $\wh{b}_i(t)=1-c_i(t)$,
$i=0,1,\ldots,n$. Also, the eigenvalues of $\Br(t)$ are symmetrically distributed with
respect to the origin (cf. \eqref{eq.BmB}). Therefore, $\la_{\max}(\Br,t)=-\la_{\min}(\Br,t)$ and so,
$\la'_{\max}(\Br,t)>0$ if and only if $\la'_{\min}(\Br,t)<0$. But $\la'_k(\wh{\Ar},t)=\la'_k(\Br,t)+1$. In fact, if
$\wh{\qr}_k(t)$ and $\wh{\Dr}(t)$ are the vector $\qr_k(t)$ of  \eqref{eq.qk} and the matrix of \eqref{eq.defS},
respectively, when $a_j(t)$ and $b_j(t)$ have been replaced by $\wh{a}_j(t)$ and $\wh{b}_j(t)$, then
$\wh{\Dr}(t)^{-1}\wh{\qr}_k(t)$ and $\wh{\Dr}(t)\wh{\qr}_k(t)$ are, for each $t\in\Is$, right and left
eigenvectors of $\wh{\Ar}(t)$ for the eigenvalue $\la_{k}(\wh{\Ar},t)$. Then, by \eqref{eq.deriveig}

\begin{align*}
\la'_{k}(\wh\Ar,t)&=\dps{\frac{\wh\qr_k(t)^\mathsf{T} \wh\Dr(t) \wh\Ar(t)\wh\Dr(t)^{-1}\wh\qr_k(t)}{\wh\qr_k(t)^\mathsf{T}\wh\qr_k(t)}}\\
&=\dps\frac{\wh\qr_k(t)^\mathsf{T} \wh\Dr(t) (I_{n+1}+\Br(t))\wh\Dr(t)^{-1}\wh\qr_k(t)}{\wh\qr_k(t)^\mathsf{T}\wh\qr_k(t)}\\
&=\dps 1+\frac{\wh\qr_k(t)^\mathsf{T} \wh\Dr(t) \Br(t)\wh\Dr(t)^{-1}\wh\qr_k(t)}{\wh\qr_k(t)^\mathsf{T}\wh\qr_k(t)}=1+\la'_k(\Br,t)
\end{align*}
because $\wh\Dr(t)\wh\qr_k(t)$ and $\wh\Dr(t)^{-1} \wh\qr(t)$ are also left and right eigenvectors of
$\Br(t)$ for $\la_k(\Br,t)$
respectively. Hence, if $\la'_{\min}(\wh\Ar,t)<0$ then $\la'_{\max}(\Br,t)=-\la'_{\min}(\Br,t)>1$. In other words,
if $\la_{\min}(\wh\Ar,t)$ decreases then $\la_{\max}(\Br,t)$ increases. It is a consequence of \eqref{eq.bjminu}
that if $(\wh{a}_j(t)\wh{b}_{j+1}(t))'>0$ for all $j=0,1,\ldots, n$ then $\la'_{\min}(\Ar,t)<0$. The proposition
follows from the fact that $\delta_j(t)=\wh{a}_j(t)\wh{b}_{j+1}(t)$.
\end{proof}

\begin{obs}\label{obs.rwm1}{\rm
\begin{itemize}
\item[(a)]  It is easily computed in Example \ref{ejerw} that $\mathsf{D}^{\uparrow}_{\max}=(0, 1/\sqrt{2})$ and 
$\mathsf{D}^{\downarrow}_{\max}=(1/\sqrt{2},1)$.
\item[(b)]  It follows from the definition of $\As^{\downarrow}_{\min}$ in \eqref{eq2}
that $\la'_{\min}(\wh\Ar,t)<0$ if $\wh{b}_0(t)=0$, $\wh{a}'_j(t)\leq 0$, and
$\wh{a}'_j(t)\wh{b}_j(t)-\wh{a}_j(t)\wh{b}'_j(t)\leq 0$ for each $j=1,\dots, n$,
provided that at least one of the inequalities is sharp (see also Observation \ref{obs.lmin}). Bearing
in mind that $\wh{b}_j(t)=1-c_j(t)$ and $\wh{a}_j(t)=c_j(t)$  and that $\la'_{\min}(\wh\Ar,t)<0$ implies
$\la'_{\max}(\Br,t)>0$, it is easily concluded that if for $t\in\Is$,
$c_0(t)=1$ and $c'_j(t)\leq 0$ for $j=1,\ldots, n$ with some of these inequalities sharp, 
then $\la'_{\max}(\Br,t)>0$. These are the sufficient conditions defining the set $\Cs^{\uparrow}_{\max}$
of \eqref{eq3} .\qed
\end{itemize}
 } \end{obs}

Proposition \ref{lrw} shows that the monotonicity of the eigenvalues of random walk matrices are
closely related to that of the eigenvalues of birth and death matrices. Actually this relationship is
much closer than one may expect at first sight. We claim that for any given random walk matrix
$\Br(t)$ there is a birth and death matrices $\Ar_w(t)$ such that the  positive eigenvalues of $\Br(t)$
are the (positive) square roots of $\Ar(t)$. Since the eigenvalues of $\Br(t)$ are
symmetric with respect to the origin, the monotonicity of the eigenvalues of $\Br(t)$ can be reduced to that of
the eigenvalues of $\Ar_w(t)$. 

To begin with, we use the notation of Proposition \ref{lrw}; that is, $\Br(t)$ is the random walk matrix of 
\eqref{eq.defB} and $\delta_{j}(t)=(1-c_{j+1}(t))c_{j}(t)$, $0\leq j \leq n-1$.  If we put 
$$
\tau_0(t)=1,\quad \tau_j(t)=\sqrt{\frac{c_0(t) \cdots c_{j-1}(t)}{(1-c_1(t))\cdots (1-c_j(t))}}, \quad j=1,\ldots, n,
$$ and
$\Tr(t)=\diag (\tau_0(t),\tau_1(t),\ldots, \tau_n(t))$ then 
\[
\Sr_w(t)=\Tr(t)\Br(t)\Tr(t)^{-1}=\begin{pmatrix}
0 & \sqrt{\delta_0} & &\\
\ \  \sqrt{\delta_0} &   0 &  \ddots & \\
& \ddots & \ddots &    \sqrt{\delta_{n-1}}  \\
& &    \sqrt{\delta_{n-1}} & 0
\end{pmatrix}.
\]
Now, set $x_{j/2}(t)=\delta_j(t)$ if $j$ is even, and $y_{(j+1)/2}(t)=\delta_j(t)$ if $j$ is odd  ($0\leq j \leq n-1$). Hence
\be\label{eq.newS}
\Sr_w(t)=\begin{pmatrix}
0 & \sqrt{x_0(t)} & & & \\
\sqrt{x_0(t)} & 0 & \sqrt{y_1(t)} & &\\
 & \sqrt{y_1(t)} & 0 & \sqrt{x_1(t)} & \\
 & & \sqrt{x_1(t)} &   0 &  \ddots   \\
& & & \ddots & \ddots   \\
\end{pmatrix}.
\ee
Since $\Br(t)$ is singular if and only if $n$ is even,  there is no loss of generality in assuming  that the 
order $n+1$ of $\Sr_w(t)$ is even and then we set $m=(n-1)/2$. By a result of Golub and Kahan
(see \cite[Section 3]{GK65}), for each $t\in\Is$,the positive eigenvalues of $\Sr_w(t)$ are the singular
values of
\[
\mathrm{J}_{m+1}(t)=
\begin{pmatrix}
\sqrt{x_0(t)} & \sqrt{y_1(t)} & &\\
\ \  &   \sqrt{x_1(t)} &  \ddots & \\
&  & \ddots &    \sqrt{y_{m}(t)}  \\
& &   & \sqrt{x_{m}(t)}
\end{pmatrix}.
\]
But these are the positive square roots of the eigenvalues of
\[
\Jr_{m+1}(t)^\mathsf{T}\Jr_{m+1}(t)=\begin{pmatrix}
x_0(t)& \sqrt{x_0(t) y_1(t)} & & \vspace*{.5em}\\
 \sqrt{x_0(t) y_1(t)} &  \ \ x_1(t)+y_1(t) & \hspace*{2em}  \ddots  \\
 &\hspace*{-6em}\ddots & \hspace*{-0.5em} \ddots & \sqrt{x_{m-1}(t)y_m(t)}\vspace*{.5em}   \\
 & & \hspace*{-3em}\sqrt{x_{m-1}(t)y_m(t)} & x_{m}(t)+y_m(t)
\end{pmatrix},
\]
which in turns is diagonally similar to (see \eqref{eq.defS})
\[
\Ar_w(t)=\begin{pmatrix}
x_0(t) & x_0(t) & &\\
\ \ y_1(t) &   x_1(t)+y_1(t) &  \ddots & \\
& \ddots & \ddots &   x_{m-1}(t)  \\
& & y_{m}(t) & x_m(t)+y_{m}(t)
\end{pmatrix}.
\]
This is a birth and death matrix for which the positive square roots of its eigenvalue-functions are the
positive eigenvalue-functions of the original random walk matrix $\Br(t)$. To the best of our knowledge, 
results about the monotonicity of the eigenvalues (other than the biggest one) of $\Br(t)$ do not exist 
in the literature. However, since for each $t\in\Is$, the eigenvalues
of $\Br(t)$ are symmetric with respect to the origin (including, possibly, an eigenvalue equal to $0$),
their monotonicity can be obtained out of the monotonicity of the eigenvalues of
$\Ar_w(t)$. In particular, we can apply all sufficient conditions studied in
Sections \ref{ext}--\ref{sec.QueQ} to $\Ar_w(t)$ in order to obtain
sufficient conditions for the eigenvalues of $\Br(t)$ to increase or decrease. As a simple example, sufficient
conditions for the eigenvalue-function $\la_{m+1}(\Br,t)$ (i.e., the smallest positive eigenvalue-function of
$\Br(t)$) to increase can be obtained from the  results in 
Observation \ref{obs.lmin}. In fact, if $0<\mu_0(\Ar_w,t)<\mu_1(\Ar_w,t)<
\cdots <\mu_m(\Ar_w,t)$ are the eigenvalue-functions of $\Ar_w(t)$ then
$\la_{m+1}(\Br,t)=\sqrt{\mu_0(\Ar_w,t)}$. Thus, $\la'_{m+1}(\Br,t)>0$ if and only if $\mu'_0(\Ar_w,t)>0$.
Taking into account that $y_0(t)=0$ in $\Ar_w(t)$, we can use  the  results in Observation \ref{obs.lmin}
to provide sufficient conditions for $\mu'_0(\Ar_w,t)>0$.  In particular, we can use the condition defining
the set $\As_{\min}^{\uparrow}$. This is:
\[
x'_0(t)>0 \text{ and } \left(x'_j(t)>0 \text{ and } \left(\frac{x_j(t)}{y_j(t)}\right)'>0,\ \ j=1,\ldots, n\right)
\]
Bearing in mind that $x_j(t)=\delta_{2j}(t)=c_{2j}(t)(1-c_{2j+1}(t))$ and $y_j(t)=\delta_{2j-1}(t)=c_{2j-1}(t)(1-c_{2j}(t))$,
these inequalities can be readily translated into inequalities involving the elements of $\Br(t)$.

\section{Conclusions}
\label{sec:conclusions}

The monotonicity of the eigenvalue-functions of finite non-homogeneous (time- dependent) birth and death matrices
and random walk matrices has been studied. New sets have been provided where they increase or decrease. Special 
attention has been paid to the extreme, maximal and minimal, eigenvalues for which, in some
cases, these sets are wider than the ones defined by Ismail in the context of birth and death or
random walk orthogonal polynomials. A key idea in this improving process is to diagonally symmetrize the given
birth and death matrix and take advantage of the properties of the eigenvalues and eigenvectors of symmetric 
matrices. By using this technique, an independent proof of a result derived from a Theorem of Magagna about the
monotonicity of homogeneous birth and death matrices has been provided. As far as random walk matrices is
concerned, it has been shown that there is a very close relationship between their eigenvalues and those of certain birth
and death matrices. This relationship allows a direct application to random walk matrices of the results about 
monotonicity of the eigenvalues of birth and death matrices.

\section*{Acknowledgments}
The authors thank the University of Toronto Archives and Records Management Services for kindly sending the a copy of Lino Magagna's Ph.D. Thesis. KC is partially supported by the Centre for Mathematics of the University of Coimbra--UID/MAT00324/ 2019, funded by the Portuguese Government through FCT/MEC and co-funded by the European Regional Development Fund through the Partnership Agreement PT2020. IZ is supported by ``Ministerio de Econom\'ia, Industria y Competitividad (MINECO)'' of Spain and ``Fondo Europeo de Desarrollo Regional (FEDER)'' of EU through grants MTM2017-83624-P and MTM2017-90682-REDT, and by UPV/EHU through grant GIU16/42.
 
\bibliographystyle{plain}
\bibliography{references}

\end{document}